\documentclass[11pt]{amsart}
\usepackage{amssymb,amsmath,epsfig}
\usepackage[usenames]{color}
\newcommand{\If}{f^o}

\newtheorem{theorem}{Theorem}[section]
\newtheorem{proposition}[theorem]{Proposition}
\newtheorem{lemma}[theorem]{Lemma}
\newtheorem{definition}[theorem]{Definition}
\newtheorem{corollary}[theorem]{Corollary}

\theoremstyle{remark}
\newtheorem{example}[theorem]{Example}

\newtheorem{conjecture}[theorem]{Conjecture}
\newtheorem{remark}[theorem]{Remark}

\DeclareMathOperator{\link}{link}
\DeclareMathOperator{\Int}{Int}
\makeatletter
\@addtoreset{equation}{section}
\makeatother

\begin{document}

\title{Generalized Tchebyshev triangulations}

\author{G\'abor HETYEI}
\address{Department of Mathematics and Statistics,
  UNC-Charlotte, Charlotte NC 28223-0001.\hfill\break
WWW: \tt http://www.math.uncc.edu/\~{}ghetyei/.}
\email{ghetyei@uncc.edu}

\author{Eran NEVO}
\address{
Einstein Institute of Mathematics,
The Hebrew University of Jerusalem,
Jerusalem, Israel and
Department of Mathematics,
Ben Gurion University of the Negev,
Be'er Sheva 84105, Israel
}
\email{nevo@math.huji.ac.il}
\thanks{
This work was partially supported by a
grant from the Simons Foundation (\#245153 to G\'abor Hetyei).
Research of the second author was partially supported by Marie Curie
grant IRG-270923 and ISF grant 805/11.
}

\date{\today}

\begin{abstract}
After fixing a triangulation $L$ of a $k$-dimensional simplex that has no
new vertices on the boundary, we introduce a triangulation operation on all
simplicial complexes that replaces every $k$-face with a copy of $L$,
via a sequence of induced subdivisions.
The operation may be performed
in many ways, but we show that the face numbers of the subdivided
complex depend only on the face numbers of the original complex, in a
linear fashion. We use this linear map to define a sequence of
polynomials generalizing the Tchebyshev polynomials of the first kind
and show, that in many cases, but not all,
the resulting polynomials have only real
roots, located in the interval $(-1,1)$. Some analogous results are
shown also for generalized Tchebyshev polynomials of the higher kind,
defined by summing over links of all original faces of a given
dimension in our generalized Tchebyshev triangulations. Generalized
Tchebyshev triangulations of the boundary complex of a cross-polytope
play a central role in our calculations, and for some of these we verify the
validity of a generalized lower bound conjecture by the second author.
\end{abstract}

\maketitle

\section{Introduction}
\setcounter{equation}{0}

This paper generalizes the following idea of a Tchebyshev triangulation
introduced in~\cite{Hetyei-Tch}: given any simplicial complex $K$,
subdivide each edge into
two parts by adding a new midpoint, and triangulate $K$ by performing a
stellar subdivision at each of the newly added midpoints. The order in
which these subdivisions have to be performed is subject to certain
rules, and then the face numbers of the resulting complex $K'$ are always
the same. The effect of this triangulation operation on the face numbers
$f_{j}$ is most easily described in terms of the {\em $F$-polynomial}
$\sum_{j\geq 0} f_{j-1} ((x-1)/2)^j$ of these complexes: the operation
taking the $F$-polynomial of $K$ into the $F$-polynomial of $K'$ is an
instance of the linear map $T: {\mathbb R}[x]\rightarrow {\mathbb R}[x]$
that takes each $x^n$ to $T_n(x)$, the $n^{\rm th}$ Tchebyshev
polynomial of the first kind.

A key result of the present paper, Theorem~\ref{thm:samef},
is a wide-reaching generalization of the idea presented above. It states
that the stellar subdivision operations above may be performed in any
order, and we always obtain the same face numbers. Furthermore, the
statement may be generalized to the situation where instead of
subdividing each edge into two parts, we subdivide each $k$-dimensional
face in the same way, using a fixed triangulation $L$ of the $k$-simplex
that adds new vertices only in the interior. The resulting {\em
  generalized Tchebyshev triangulations} are the subject of study of our
present paper.

As we will see in Section~\ref{sec:genf}, the face numbers in a
generalized Tchebyshev triangulation can be easily computed knowing the number of faces of $L$ with given numbers of vertices on the
boundary and in the interior of the $k$-simplex. At the level of the
$F$-polynomials, each fixed subdivision $L$ induces a linear map
$T^L: {\mathbb R}[x]\rightarrow {\mathbb R}[x]$, giving rise to a
natural generalization of Tchebyshev polynomials of the first kind,
introduced in Section~\ref{sec:TL}. These polynomials share many
properties with the ordinary Tchebyshev polynomials: they satisfy a
Fibonacci-type recurrence (whose degree depends on the dimension $k$),
their multiset of zeros is symmetric around the origin, and all their real
zeros belong to the interval $(-1,1)$. The question naturally arises,
whether these generalized Tchebyshev polynomials of the first kind also
have only real roots. A first answer to this question is given in
Section~\ref{sec:rr}, where we will see that the answer is always ``yes''
for $k=1$, and it is ``no'' for the simplest subdivision of a
$3$-simplex, obtained by adding one new vertex in the
interior and performing a stellar subdivision. In Section~\ref{sec:pn}
we prove a general real-rootedness result for a class of polynomial sequences,
which implies that all roots are real also for generalized Tchebyshev
polynomials of the first kind, induced by any valid subdivision of the two-dimensional simplex.

Generalizing the construction introduced in \cite{Hetyei-Tch}, in
Section~\ref{sec:higher} we introduce analogues of Tchebyshev
polynomials of the second kind, by considering summing over the links of
all faces of a given dimension of the original complex, in the
subdivided complex. A lot remains to be explored regarding these polynomials,
but a few results indicate that we have found an ``appropriate
generalization'': our generalized Tchebyshev polynomials of the $j^{\rm
  th}$ kind (where $j\leq k+1$) satisfy the same recurrence as our
generalized Tchebyshev polynomials of the first kind, the multisets of
their roots are also symmetric of the origin, and their real roots also
belong to the interval $(-1,1)$. We chose to postpone a deeper study of
their real-rootedness to a future occasion, but we established the fact
that, for $k=1$, all generalized Tchebyshev polynomials of the second
kind are real rooted.

Our results in Sections~\ref{sec:TL} and \ref{sec:higher} underline the
central importance of the generalized Tchebyshev triangulations of the
boundary complex of a cross-polytope, as the coefficients of our
generalized Tchebyshev polynomials can be directly read off the face
count in these complexes, refined by distinguishing between original and
newly added vertices; see the important Corollary~\ref{cor:Tcross}. In the concluding Section~\ref{sec:glbf} we prove
the validity of a conjecture by the second author~\cite[Conjecture
  1.5]{Nevo-Missing}, on strong generalized lower bounds for the face
numbers of some of these simplicial complexes.

Our generalized Tchebyshev triangulations offer infinitely
many new ways to subdivide a simplicial complex in such a manner that
the face numbers change in a predictable fashion. In this sense our
triangulation operations generalize the notion of a barycentric
subdivision. In fact, any barycentric subdivision arises by applying a
sequence of generalized Tchebyshev triangulation operations as follows:
for each $k$ that is less than or equal to the dimension of the complex
to be subdivided, we take the generalized Tchebyshev triangulation induced
by the stellar subdivision of a $k$-simplex obtained after adding a
single vertex in its interior (we perform these operation in decreasing order by
$k$). Investigating whether some face counting polynomial
associated to such a triangulation has only real roots is not a new
concern: Brenti and Welker~\cite{Brenti-Welker} showed that the
$h$-polynomial of the barycentric subdivision of a simplicial complex
with a nonnegative $h$-vector has only simple real zeros. In the future,
it would be worth finding an exact description of all triangulations of a
$k$-simplex that induces generalized Tchebyshev polynomials having only real
roots. Another interesting question is to fix a specific
generalized Tchebyshev triangulation operation, and to ask: to which simplicial
complexes can we apply them and obtain real-rooted $f$-polynomials
and/or $h$-polynomials? Finally, once we have a better understanding of
the generalized Tchebyshev polynomials of the higher kind, it will
be worth finding out how they are interconnected.
%The easiest nontrivial
%example worth considering is the generalized Tchebyshev polynomials
%associated to the triangulation of a $2$-dimensional triangle obtained
%by performing a single stellar subdivision at a new vertex added in the
%interior.
%%%%%%%%%%%%%%%%%%%%%%%%%%%%%%%%%%%%%%%%%%%%%%%%%%%%

\section{Preliminaries}
\setcounter{equation}{0}
First we recall some basic definitions and results related to simplicial complexes. For further background see, for instance, \cite{BjornerTopMeth,Munkres}.
Next we recall some basic facts on Tchebyshev polynomials. These
polynomials play an important role in many areas of mathematics,
including combinatorics, numerical analysis and orthogonal
polynomials. However, we will only need facts on them that are
discussed in any introductory work on orthogonal polynomials, see for
instance~\cite{Chihara}. Most important formulas on Tchebyshev
polynomials are listed (without proof) in the work of Abramowitz and
Stegun~\cite{Abramowitz-Stegun}.

\subsection{Simplicial complexes.}
A \emph{simplicial complex} $K$ on the \emph{vertex set} $V$
is a collection of subsets of $V$ such that
$\{v\}\in K$ for all $v\in V$, and
if $G\subset F$ and $F\in K$, then $G\in K$. The elements of $K$
are called {\em faces}. In particular, the empty set is a face of $K$.
The \emph{link} of a face $\sigma$ is the subcomplex $\link_K(\sigma)=\{\tau\in K:\
\sigma\cap \tau=\emptyset, \ \sigma\cup \tau\in K\}$.
The \emph{join} of two simplicial complexes $K$ and $L$ on disjoint
vertex sets is $K*L=\{\sigma\cup\tau:\ \sigma\in K,\ \tau\in L\}$.
Thinking of the faces of $K$ as simplices glued together gives $K$ a
topology, and the \emph{geometric realization} of $K$, denoted $\|K\|$, stands for this topological space.
We say that $K$ is a \emph{triangulation} of a topological space $X$
if $\|K\|$ is homeomorphic to $X$.

The following well known result in piecewise linear topology will be
needed later, see e.g. \cite[Cor.\ 1.16-- Lemma 1.18]{Hudson}:
\begin{lemma}
\label{l:link} Let $L$ be a triangulation of a simplex
such that the only vertices of $L$ on the boundary $\partial(L)$ are the
original vertices of the simplex, and let $\tau\in L$ be any face. Then
$\link_L(\tau)$ is homeomorphic to a sphere if and only if it contains at
least one vertex in the interior of $\|L\|$, otherwise it is homeomorphic to a ball.
\end{lemma}

Let $(A_i)_{i\in I}$ be a family of nonempty sets. Its \emph{nerve}
$\mathcal{N}((A_i)_I)$ is the simplicial complex with vertex set $I$ and
faces all $F\subseteq I$ such that $\cap_{i\in F}A_i\neq \emptyset$. A
version of the Borsuk nerve theorem \cite{Borsuk48NerveThm} that we will need is the following, see Bj\"{o}rner \cite[Theorem 10.6]{BjornerTopMeth})
\begin{theorem}[Nerve theorem]\label{thm:nerve}
Let $(A_i)_{i\in I}$ be a family of subcomplexes of a simplicial complex $K$ such that $\cup_I A_i = K$ and for every $J\subseteq I$, $\cap_{i\in J}A_i$ is either empty or contractible. Then the nerve complex $\mathcal{N}((A_i)_I)$ is homotopy equivalent to $K$.
\end{theorem}

The \emph{dimension} of a face $\sigma$ is defined by
$\dim(\sigma):=|\sigma|-1$; the dimension of a simplicial complex
$K$ is defined by $\dim(K):=\max\{ \dim(\sigma) \ : \ \sigma\in K\}$.
Let $f_i(K)$ be the number of
$i$-dimensional faces (\emph{$i$-faces}) of $K$, and let $f(K)$
be the \emph{$f$-vector} of $K$, namely, $f(K):=(f_{-1}(K),f_0(K),\ldots,f_{\dim(K)}(K))$.
In polynomial form, the \emph{$f$-polynomial} of $K$ is
$f(K,x):=\sum_{0\leq i\leq \dim(K)+1}f_{i-1}(K)x^i$.
This information can also be encoded in the \emph{$h$-polynomial} of $K$, $h(K,x):=\sum_{0\leq i\leq \dim(K)+1}h_{i}(K)x^i$, given by
$h_i=\sum_{j=0}^i (-1)^{i-j}\binom{n-j}{i-j}f_{j-1}$ where $n=\dim(K)+1$. In particular, $f_{i-1}=\sum_{j=0}^i\binom{n-j}{i-j}h_{j}$.
Given a simplicial complex $K$ and a map (called \emph{coloring})
$a:V(K)\rightarrow \{x_1,x_2,\ldots ,x_s\}$, the \emph{flag $f$-polynomial} of $(K,a)$ is
$$f_a(K;x_1,\ldots ,x_s):=\sum_{F\in K}\prod_{v\in F}a(v)\in \mathbb{Z}[x_1,\ldots,x_s].$$

A set $F$ is called a \emph{missing face} of a simplicial complex $K$ if $F\notin K$ and its boundary complex $\partial(F)=2^F\setminus \{F\}$ is a subcomplex of $K$.
For $F\in K$, the \emph{stellar subdivision} of $K$ at $F$ is the simplicial complex $K(F)=\operatorname{Stellar}_K(F)=(K\setminus \{T\in K:\ F\subseteq T\}) \cup \{v_F\}* \partial (F) * \link_K(F)$, where $v_F$ is not a vertex of $K$. The \emph{$j$-skeleton} of $K$, denoted $K_{\leq j}$, is the subcomplex of $K$ consisting of all faces in $K$ of dimension $\leq j$.

\subsection{Tchebyshev polynomials}
The Tchebyshev polynomials $T_n(x)$ of the first kind and the
Tchebyshev polynomials $U_n(x)$ of the second kind are usually defined by the
formulas
\begin{equation}
\label{eq:Ttrigdef}
T_n(\cos(\alpha))=\cos(n\cdot \alpha)\quad\mbox{and}\quad
U_n(\cos(\alpha))=\frac{\sin((n+1)\alpha)}{\cos(\alpha)},
\end{equation}
see \cite[(22.3.15) and (22.3.16)]{Abramowitz-Stegun}. Equivalently,
they may be defined recursively as follows. The polynomial sequences
$\{T_n(x)\}_{n=0}^{\infty}$ and $\{U_n(x)\}_{n=0}^{\infty}$
are the unique solutions to the recurrence $P_n(x)=2x
P_{n-1}(x)-P_{n-2}(x)$ (all occurrences of the letter $P$ need to be
                            replaced by $T$ resp.\ $U$,
see \cite[(22.7.4) and (22.7.5)]{Abramowitz-Stegun}), subject to the
initial conditions $T_0(x)=1$, $T_1(x)=x$ and  $U_0(x)=1$, $U_1(x)=2x$.
%Both polynomial sequences
%satisfy the same recurrence  $P_n(x)=2x P_{n-1}(x)-P_{n-2}(x)$ (all
%occurrences of the letter $P$ need to be replaced by $T$ resp.\ $U$),
%see \cite[(22.7.4) and (22.7.5)]{Abramowitz-Stegun}.  Only the
%initial conditions are different: we have $T_0(x)=1$, $T_1(x)=x$;
%$U_0(x)=1$ and $U_1(x)=2x$.

They share the following properties, which will be explored for our generalized  Tchebyshev polynomials.

\begin{theorem}\label{thm:basicsTU}
For all $n\geq 0$, the polynomials $T_n(x)$ and $U_n(x)$ have the
following properties:
\begin{enumerate}
\item their degree is $n$,
\item they are symmetric in the sense that  $(-1)^n P_n(-x)=P_n(x)$
  holds for $P_n=T_n,U_n$, and
\item all their roots are real, simple, and belong to the interval $(-1,1)$.
\end{enumerate}
\end{theorem}
The first two statements are immediate consequences of the recursive
definition, the third statement may be shown in at least two different
ways: by direct calculation of the roots from (\ref{eq:Ttrigdef}), or
by invoking general results from the theory of orthogonal
polynomials. We refer the reader to \cite{Chihara} for
details which we will not review here as most of our generalized
Tchebyshev polynomials will {\em not} be sequences of orthogonal
polynomials, see Remark~\ref{rem:notorth}.

%%%%%%

\section{Generalized Tchebyshev triangulations}
\label{sec:gtch}

In this section,
and throughout the rest of the paper,
we fix a triangulation $L$ of the $k$-dimensional
simplex such that the only vertices of $L$ on the boundary $\partial(L)$
are the original vertices of the simplex. We will use the notation
$\partial(L)$ for the subcomplex of boundary faces (this is
also the boundary complex of the original $k$-simplex) and the
notation $\Int(L)$ for the family of (closed) faces contained in the interior of
$L$. Given any family of faces $C$ we will use $V(C)$ to denote the set
of vertices of the faces in the family $C$. Any face $\sigma\in L$ may
be uniquely written as the disjoint union of $\sigma\cap V(\partial(L))$
and $\sigma\cap V(\Int(L))$.
\begin{definition}
\label{def:gtch}
%Let $L$ be a triangulation of the $k$-simplex, containing new vertices
%in the interior only.
Given any simplicial complex $K$, a simplicial complex $K'$ is a
{\em generalized Tchebyshev triangulation of $K$ induced by $L$} if
there is an ordered list $\sigma_1,\ldots,\sigma_m$ of the
$k$-dimensional faces of $K$, listing each $k$-dimensional face exactly
once, and an ordered list $K_0,K_1,\ldots,K_m$ of
simplicial complexes such that $K_0=K$, $K_m=K'$ and, for each $i\geq
1$, the complex $K_{i}$ is obtained from $K_{i-1}$ by replacing the face
$\sigma_i$ with an isomorphic copy $L_i$ of $L$ and the family of faces
$\{\sigma_i\cup\tau\:: \tau\in \link(\sigma_i)\}$ containing $\sigma_i$
with the family of faces $\{\sigma'\cup \tau \:: \sigma'\in
L_i, \tau\in \link(\sigma_i)\}$.
In short, replace the closed star $\overline{\sigma_i}*\link(\sigma_i)$ by the complex $L_i * \link(\sigma_i)$.
\end{definition}
Obviously, given any ordered list $\sigma_1,\ldots,\sigma_m$ of the
$k$-dimensional faces of $K$,
%up to isomorphism
and a list of bijections $\phi_{\sigma_i}:V(\partial(L))\rightarrow V(\sigma_i)$ for $1\leq i\leq m$,
 there is exactly one
list of simplicial complexes $K_0,K_1,\ldots,K_m$ satisfying the
condition given in Definition~\ref{def:gtch}. Using a different list may
result in a non-isomorphic triangulation, as shown in the
following example.
\begin{example}
\label{ex:2triang}
Let $L$ be the path with $2$ edges
(triangulating the $1$-simplex), $K$ be the union of the two triangles
$\{v_1, v_2, v_3\}$ and $\{v_1,v_2, v_4\}$ sharing the edge $\{v_1,
v_2\}$. Let $K'$ be the generalized Tchebyshev triangulation of $K$
defined by the ordering of edges $\{v_1,v_2\}$, $\{v_1,v_3\}$, $\{v_2, v_3\}$, $\{v_1, v_4\}$, $\{v_2,v_4\}$ and let $K''$ be the generalized
Tchebyshev triangulation of $K$ defined by the ordering $\{v_1,v_3\}$,
$\{v_2,v_3\}$, $\{v_2,v_4\}$, $\{v_1,v_4\}$, $\{v_1,v_2\}$.
Then $K'$ is a cone (over an $8$-cycle) and $K''$ is not, see
Fig.\ \ref{fig:2triang}.
(Here specifying the bijection $\phi_{\sigma_i}$ is not important as we obtain isomorphic complexes for both choices.)
\end{example}
\begin{figure}[h]
\begin{center}
\input{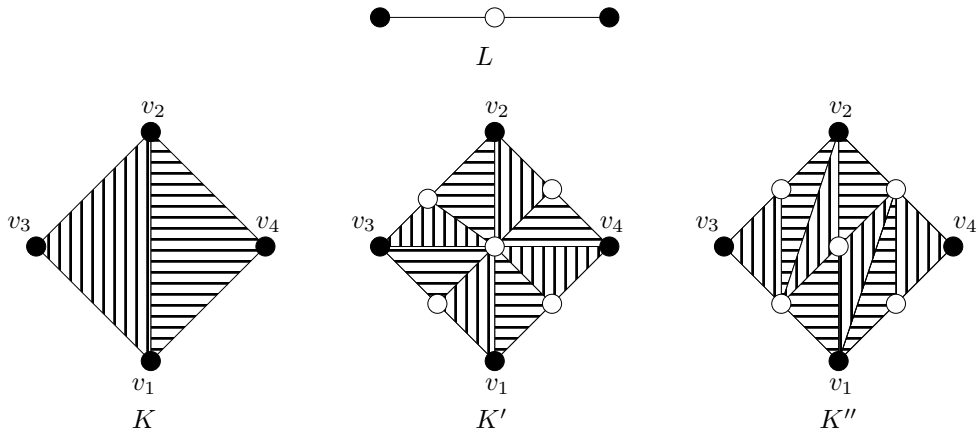}
\end{center}
\caption{Illustration to Example \ref{ex:2triang}}
\label{fig:2triang}
\end{figure}

%Since a generalized Tchebyshev triangulation $K'$ is a triangulation of
%the original complex $K$, the two complexes are homeomorphic. Moreover, we have the following result.

However, $K'$ and $K''$ have the same $f$-vector. This is not a coincidence as the following result shows.

\begin{theorem}
\label{thm:samef}
Given
%a triangulation $L$ of the $k$-dimensional simplex and
an
arbitrary simplicial complex $K$, all
generalized Tchebyshev triangulations of $K$, induced by $L$, have the
same $f$-vector.
\end{theorem}
Theorem~\ref{thm:samef} follows from setting $y=x$ in the following,
more general statement.

Let $c=c_K:V(K')\rightarrow \{x,y\}$ be the coloring $c(v)=x$ if $v\in V(K)$ and $c(v)=y$ if $v\in V(K')-V(K)$.

\begin{theorem}
\label{thm:samefxy}
%Given a triangulation $L$ of the $k$-dimensional simplex,
Let $K'$ be
any generalized Tchebyshev triangulation of $K$ induced by $L$. Then the flag $f$-polynomial of $(K',c)$, namely
$$
f_c(K';x,y)=\sum_{\sigma\in K'} x^{|\sigma\cap V(K)|}y^{|\sigma\cap
  (V(K')\setminus V(K))|},$$
depends only on the $f$-vector of $K$ in a linear fashion and is
independent of the particular choice of $K'$. Thus, we will denote it by
$f(K;x,y)$.
\end{theorem}
\begin{proof}
We need to show that given $L$, there exist linear functionals $l_{i,j}:\mathbb{R}[z]\rightarrow \mathbb{R}$ such that for any simplicial complex $K$ and any generalized Tchebyshev triangulation $K'$ of $K$ induced by $L$, one has
$$f_c(K';x,y)=\sum_{i,j}l_{i,j}(f(K,z))x^i y^j
.$$

\emph{Step 1 -- simplex case:}
First we prove for the case where $K$ is a simplex.

Assume $K$ is an $(n-1)$-dimensional simplex. We will show now by induction on $n$ that $f_c(K';x,y)$ does not depend on the choice of $K'$, only on $n$,
in which case we use $f_n(x,y)$ to stand
for $f_c(K';x,y)$. We claim that $f_n(x,y)$ is given by $f_n(x,y)=(1+x)^n$
for $n\leq k$ and by the following recurrence formula, which shows the independence of the choice of $K'$:
\begin{equation}
\label{eq:fnxy}
f_n(x,y)= \sum_{\emptyset\neq {\mathcal I}\subseteq {\mathcal F}(L)}
(-1)^{|{\mathcal I}|-1}(1+y)^{N(\mathcal I )}
f_{n-k-1+O(\mathcal I )}(x,y)
\quad\mbox{for $n\geq k+1$.}
\end{equation}
Here $\mathcal F(L)$ is the set of facets of $L$,
$N(\mathcal I):=|\bigcap_{F \in {\mathcal I}} F\cap V(\Int(L))|$, $O(\mathcal I):=|\bigcap_{F \in {\mathcal I}} F\cap V(\partial(L))|$,
and the summation runs
over all nonempty subfamilies ${\mathcal I}$ of $\mathcal F(L)$.
%The case $n=k$ is clear from the definition and inclusion-exclusion.

To prove (\ref{eq:fnxy}),
for $n>k$, we argue by induction on $n$: the subdivision of $\sigma_1$ induces a bijection $\iota$ from
$\mathcal F(\sigma_1')$ to $\mathcal F(K_1)$ by $\iota(F)=F\cup (V(K)-V(\sigma_1))$, for any facet $F$ of $\sigma_1$.

Denote by $\overline{U}=2^U$ the simplex on the finite vertex set $U$ and all its faces, and for a family $\mathcal{U}$ of finite sets let $\overline{\mathcal{U}}$ be the simplicial complex $\cup_{U\in \mathcal{U}}\overline{U}$.
For a face $\sigma\in K'$ let $N(\sigma):=\sigma\cap (V(K')\setminus
V(K))$ and $O(\sigma)=\sigma\cap V(K)$. (The letters N and O in
$N(\sigma)$ and $O(\sigma)$, respectively, are meant to refer to
``new'', respectively, ``old'' vertices.)

We will make use of the following easy observation:

(*) If $H$ is a subcomplex of $K$, then the induced order on the $k$-simplices in $H$
(keeping the same bijections $\phi_{\sigma}:V(\partial(L))\rightarrow V(\sigma)$ for all $k$-faces $\sigma\in H$)
gives $H'$ which equals to the restriction of $K'$ to the subspace
$\|H\|$ of $\|K\|$.
In particular, by restriction, $c_K$ induces a coloring of $V(H')$ which is the same as $c_H:V(H')\rightarrow \{x,y\}$.

Now, observe that for $F\in\mathcal F(\sigma_1')$  the restriction of $K'$ to $\|\overline{\iota(F)}\|$ is the subcomplex $(\overline{\iota(F)})'= \overline{N(F)}*(\ \overline{V(K)-V(\sigma)}*\overline{O(F)}\ )'$.
Multiplying the flag $f$-polynomials of each part in this join,
and applying the induction hypothesis to the right hand part,
we obtain that the contribution of $(\overline{\iota(F)})'$
to the flag $f$-polynomial $f_c(K';x,y)$ is
$(1+y)^{|N(F)|}f_{n-|N(F)|}(x,y)$. (Note that indeed $|N(F)|>0$ as all facets of $\sigma_1$ contain a vertex in $\Int(\sigma_1)$.)

As $K'=\cup_{F\in \mathcal F(\sigma_1')}(\overline{\iota(F)})'$, inclusion-exclusion gives (\ref{eq:fnxy}).
Here are the details:
\begin{multline*}
f_n(x,y)=
\sum_{\emptyset \neq S \subseteq \mathcal F(\sigma_1')}
(-1)^{|S|-1} f_c\left(\bigcap_{F\in S} \overline{\iota(F)}';x,y\right)=
\\
\sum_{\emptyset \neq S \subseteq \mathcal F(\sigma_1')}
(-1)^{|S|-1} f_c\left(\bigcap_{F\in S}\overline{N(F)} * (\bigcap_{F\in S} \overline{O(F) * (V(K)-V(\sigma_1))}\ )' ;x,y\right)
.
\end{multline*}
Now, $f_c(\cap_{F\in S}\overline{N(F)};x,y)=(1+y)^{|\cap_{F\in S}N(F)|}=(1+y)^{N(S)}$
and, by induction,
$$f_c\left(\bigcap_{F\in S} (\ \overline{O(F) * (V(K)-V(\sigma_1))}\ )' ;x,y\right)=
f_{|V(K)|-|V(\sigma_1)|+|\cap_{F\in S}O(F)|}(x,y)=
f_{n-(k+1)+O(S)}(x,y).$$
The bijection $\iota$ finishes the proof, when $K$ is a simplex.

\emph{Step 2 -- the general case:} We proceed by induction on dimension,
% and the number of top dimensional faces.
the case $\dim(K)<\dim(L)$ is trivial.
Using the observation (*) again, if $F$ is a top dimensional face of $K$ then
$$f_c(K';x,y)=f_c((K-\{F\})';x,y) + f_c(F';x,y) - f_c((\partial F)';x,y).$$
Note that if $G$ is another top dimensional face of $K$, then, by Step 1,
$f_c(F';x,y)=f_c(G';x,y)$, and by induction on dimension,
$f_c((\partial F)';x,y)=f_c((\partial G)';x,y)$.

Hence, by repeating for all top dimensional faces of $K$, we obtain for $d=\dim(K)$ that
$$f_c(K';x,y)=f_c((K_{\leq d-1})';x,y) + f_{d}(K) ( f_c(F';x,y) - f_c((\partial F)';x,y)).$$
By induction on dimension, we already know that there exist linear functionals $l_{i,j}^{(d)}:\mathbb{R}[z]_{\leq d}\rightarrow \mathbb{R}$ such that for all complexes $T$ of dimension $<d$, $f_c(T';x,y)=\sum_{i,j}l_{i,j}^{(d)}(f(T,z))x^i y^j$.
Now, using the assertion of the theorem for the $d$-simplex (see Step 1), $l_{i,j}^{(d)}$ can be extended to
$l_{i,j}^{(d+1)}:\mathbb{R}[z]_{\leq d+1}\rightarrow \mathbb{R}$
by setting
$l_{i,j}^{(d+1)}(z^{d+1})$ to be the coefficient of $x^iy^j$ in $f_c(F';x,y) - f_c((\partial F)';x,y)$.
\end{proof}

\begin{example}
\label{ex:Tch}
Let $k=1$ and let $L$ be the triangulation of the $1$-dimensional
simplex obtained by adding the midpoint of the $1$-simplex as a new
vertex, as in Example \ref{ex:2triang}. Certain generalized Tchebyshev
triangulations induced by this complex $L$ were considered
in~\cite{Hetyei-Tch}, where it was shown that the face numbers in these
triangulations are independent of the numbering of the vertices.
Theorem~\ref{thm:samef} generalizes these results even for this
particular choice of $L$.
\end{example}

\begin{remark}
Theorems~\ref{thm:samef} and~\ref{thm:samefxy} are true also when $K$ is a \emph{simplicial poset} (see \cite{Stanley-greenbook} for an exposition), and the proofs go through verbatim.
\end{remark}
Using Lemma~\ref{l:link} we may rephrase \eqref{eq:fnxy} as follows.
\begin{proposition}
\label{prop:fnxy}
The polynomials $f_n(x,y)$ are also given by $f_n(x,y)=(1+x)^n$ for
$n\leq k$ and by the recurrence
$$
f_n(x,y)=\sum_{\sigma\in L\setminus\partial(L)}
(-1)^{k+1-|\sigma|} (1+y)^{|\sigma\cap V(\Int(L))|}
                       \cdot f_{n-k-1+|\sigma\cap V(\partial(L))|} (x,y)\quad\mbox{for $n\geq
                         k+1$.}
$$
\end{proposition}
\begin{proof}
Let us fix a face $\sigma\in L$ and consider only those subsets ${\mathcal
  I}\subseteq {\mathcal F}(L)$ for which we have
\begin{equation}
\label{eq:st}
\sigma=\bigcap_{F \in {\mathcal I}} F.
\end{equation}
Note that ${\mathcal I}\neq \emptyset$ is equivalent to $\sigma\in L$.
Each ${\mathcal I}\subseteq {\mathcal F}(L)$ satisfying \eqref{eq:st}
contributes a term of the form
$(-1)^{|{\mathcal I}|-1} (1+y)^{|\sigma\cap V(\Int(L))|}\cdot f_{n-k-1+|\sigma\cap V(\partial(L))|} (x,y)$
to the right hand side of \eqref{eq:fnxy}.

Assume first that $\sigma$ is not a facet of $L$. Then we may
identify each ${\mathcal I}$ with the collection of
facets ${\mathcal I}':=\{F\setminus \sigma \::\: F\in {\mathcal
  I}\}$ of $\link(\sigma)$. Condition \eqref{eq:st} is then
equivalent to requiring that the intersection of the facets of
$\link(\sigma)$ listed in ${\mathcal I}'$ is empty, equivalently
the vertex set ${\mathcal I}'$ is
not a face of the
nerve complex ${\mathcal N}(\link(\sigma))$ of
$\link(\sigma)$. In this case the total contribution
of all families  ${\mathcal I}\subseteq {\mathcal F}(L)$ satisfying
\eqref{eq:st} to the right hand side of \eqref{eq:fnxy} is
$$
\left(\sum_{{\mathcal I}'\subseteq V({\mathcal
    N}(\link(\sigma)))} (-1)^{|{\mathcal I}'|-1}
-\widetilde{\chi}({\mathcal N}(\link(\sigma)))\right)\cdot
(1+y)^{|\sigma\cap V(\Int(L))|}\cdot f_{n-k-1+|\sigma\cap
  V(\partial(L))|} (x,y).$$
Here $\widetilde{\chi}({\mathcal N}(\link(\sigma)))$ is the
reduced Euler characteristic of ${\mathcal N}(\link(\sigma))$
which is the same as $\widetilde{\chi}(\link(\sigma))$ by
Borsuk's nerve theorem, Theorem \ref{thm:nerve}.
%%%%%%5
%\cite{Borsuk48NerveThm} (for more on the nerve theorem see e.g. \cite[Theorem 10.6]{BjornerTopMeth}).
%%%%%%%%5
By Lemma~\ref{l:link},
$\widetilde{\chi}(\link(\sigma))$ is nonzero exactly when
$\sigma\cap V(\Int(L))\neq \emptyset$ and then it is $(-1)^{k-|\sigma|}$,
the reduced Euler characteristic of a $(k-|\sigma|)$-dimensional sphere.
Since  $\sigma$ is not a facet of $L$, we have $|V({\mathcal
    N}(\link(\sigma)))|\geq 1$ and the sum $\sum_{{\mathcal I}'\subseteq
  V({\mathcal
    N}(\link(\sigma)))} (-1)^{|{\mathcal I}'|-1}$ is zero.

Finally, consider the case when $\sigma$ is a facet of $L$. Then
\eqref{eq:st} holds for exactly one ${\mathcal I}\subseteq
      {\mathcal F}(L)$, namely the family ${\mathcal
        I}=\{\sigma\}$. The contribution of this ${\mathcal I}$
      to the right hand side of \eqref{eq:fnxy} is
$$
(-1)^{|{\mathcal I}|-1}\cdot (1+y)^{|\sigma\cap V(\Int(L))|}\cdot
      f_{n-k-1+|\sigma\cap V(\partial(L))|} (x,y)$$
which equals
$$
(-1)^{k+1-|\sigma|}\cdot (1+y)^{|\sigma\cap V(\Int(L))|}\cdot f_{n-k-1+|\sigma\cap V(\partial(L))|}
      (x,y),
$$
since $|{\mathcal I}|-1=k+1-|\sigma|=0$.
\end{proof}

\section{A generating function for the polynomials $f_n(x,y)$}
\label{sec:genf}

The recurrence given in Proposition~\ref{prop:fnxy} allows to write a
generating function formula for the polynomials $f_n(x,y)$. To state it
in a more concise fashion we introduce the {\em magic polynomial
  $r_L(u,v)$ of the simplicial complex $L$}  given by

\begin{equation}
\label{eq:magic}
r_L(u,v)
=\sum_{\sigma \in L\setminus \partial L}
  u^{|\sigma\cap V(\partial L)|} v^{|\sigma\cap V(\Int(L)|} -u^{k+1}.
\end{equation}

\begin{proposition}
\label{prop:fgen}
The generating function $f(x,y,t):=\sum_{n=0}^{\infty} f_n(x,y) t^n$ is
given by
$$
f(x,y,t)=\frac{1}{1-t(x+1)} \left(1- \frac{r_L(-1-x,-1-y)}{r_L(-1/t, -1-y)}\right).
$$
\end{proposition}
\begin{proof}
Proposition~\ref{prop:fnxy} may be rewritten as
\begin{align*}
f(x,y,t)
&=
\sum_{n=0}^k (1+x)^n\cdot t^n \\
&+\sum_{\sigma \in L\setminus \partial L}
(-1)^{k+1-|\sigma|} (1+y)^{|\sigma\cap V(\Int (L))|}t^{k+1-|\sigma\cap
  V(\partial(L)|} \sum_{n=k+1}^{\infty} f_{n-k-1+|\sigma\cap V(\partial L)|}(x,y)
t^{n-k-1+|\sigma\cap V(\partial L)|}\\
&=\frac{1-(1+x)^{k+1}t^{k+1}}{1-(1+x)t}\\
&+\sum_{\sigma \in L\setminus \partial L}
(-1)^{k+1-|\sigma|} (1+y)^{|\sigma\cap V(\Int (L))|}t^{k+1-|\sigma\cap
  V(\partial(L)|}
\left(f(x,y,t)-\sum_{n=0}^{|\sigma\cap V(\partial L)|-1} (1+x)^nt^n\right).
\end{align*}
After subtracting $\sum_{\sigma \in L\setminus \partial L}
(-1)^{k+1-|\sigma|} (1+y)^{|\sigma\cap V(\Int (L))|}t^{k+1-|\sigma\cap
  V(\partial(L)|}f(x,y,t)$ on both sides, the left hand side becomes
$$
\left(1- \sum_{\sigma \in L\setminus \partial L}
(-1)^{k+1-|\sigma|} (1+y)^{|\sigma\cap V(\Int (L))|}t^{k+1-|\sigma\cap
  V(\partial(L)|}\right)f(x,y,t)
=-(-t)^{k+1} r_L\left(-\frac{1}{t},-1-y\right) f(x,y,t),
$$
and the right hand side becomes
$$
\frac{1-(1+x)^{k+1}t^{k+1}}{1-(1+x)t}-t^{k+1}\sum_{\sigma \in L\setminus
  \partial L}
(-1)^{k+1-|\sigma|} (1+y)^{|\sigma\cap V(\Int (L))|}\left(\frac{1}{t}\right)^{|\sigma\cap
  V(\partial(L)|} \frac{1-((1+x)t)^{|\sigma\cap V(\partial L)|}}{1-(1+x)t}
$$
which is easily seen to be equal to
$$
\frac{-(-t)^{k+1} r_L(-1/t,-1-y)+(-t)^{k+1}r_L(-1-x,-1-y)}{1-(1+x)t}
$$
Dividing both sides by $-(-t)^{k+1} r_L(-1/t,-1-y)$ yields the stated equality.
\end{proof}

Let $\If_n(x,y)$ denote the contribution to $f(K;x,y)$ of adding a
single facet $\sigma$ of dimension $(n-1)$.
Note that indeed the complexes $K$ and $K\cup\{\sigma\}$ satisfy that the polynomial $f(K\cup\{\sigma\};x,y)-f(K;x,y)$ depends only on $\dim \sigma$ (and $L$), and not on $K$.
Knowing $\If_n(x,y)$ allows
to express $f(K;x,y)$ directly since we have
\begin{equation}
\label{eq:if}
f(K;x,y)=\sum_{j=0}^{\dim (K)+1} f_{j-1}(K)\If_j (x,y).
\end{equation}
Applying \eqref{eq:if} to the case when $K$ is the
$(n-1)$-dimensional simplex yields
\begin{equation}
\label{eq:fIf}
f_n(x,y)=\sum_{j=0}^{n} \binom{n}{j}\If_j (x,y).
\end{equation}
As an immediate consequence we obtain the generating
function identity
$$
\sum_{n=0}^{\infty} f_n(x,y)t^n=\sum_{j=0}^{\infty} \If _j(x,y)
t^j\sum_{k=0}^{\infty} \binom{k+j}{j} t^k=\sum_{j=0}^{\infty} \If _j(x,y)
\frac{t^j}{(1-t)^{j+1}}.
$$
Substituting $t:=u/(1+u)$ in the previous equation and rearranging
yields
$$
\sum_{j=0}^{\infty} \If _j(x,y) u^j=\frac{1}{1+u}\sum_{n=0}^{\infty}
f_n(x,y) \left(\frac{u}{1+u}\right)^n=\frac{1}{1+u}
f\left(x,y,\frac{u}{1+u}\right).
$$
This equation and Proposition~\ref{prop:fgen} have the following consequence.

\begin{corollary}
\label{cor:ifgen}
The generating function $\If(x,y,t):=\sum_{n=0}^{\infty}\If_n(x,y) t^n$
is given by
$$
\If(x,y,t)=
\frac{1}{1-tx}\left(1-\frac{r_L(-1-x,-1-y)}{r_L\left(\frac{-1-t}{t},-1-y\right)}\right).
$$
\end{corollary}

\section{Generalized Tchebyshev polynomials of the first kind}
\label{sec:TL}

Following~\cite{Hetyei-Tch} we define the {\em $F$-polynomial} of a simplicial
complex $K$ as the polynomial
$$
F(K,x):=\sum_{j=0}^{\dim K+1} f_{j-1}(K)\left(\frac{x-1}{2}\right)^j.
$$
Let $L$ be the simplicial complex considered in Example~\ref{ex:Tch}.
As an immediate generalization of~\cite[Proposition 3.3]{Hetyei-Tch},
 Theorem \ref{thm:samefxy} gives the following.
\begin{proposition}
\label{prop:T1}
Let $L$ be the path with two edges.
Let $K$ be any simplicial complex and $K'$ be any generalized Tchebyshev triangulation of $K$ induced by $L$.
Let $T_n(x)$ be the $n$-th Tchebyshev polynomial
of the first kind.
Then the linear map $T:{\mathbb R}[x]\rightarrow
{\mathbb R}[x]$ given by $T(x^n):=T_n(x)$ satisfies
$$
T(F(K,x))=F(K',x).
$$
\end{proposition}
Inspired by Proposition~\ref{prop:T1} we make the following definition.
\begin{definition}
\label{def:T1}
%Let $L$ be a triangulation of the $k$-dimensional
%simplex such that the only vertices of $L$ on the boundary $\partial(L)$
%are the original vertices of the simplex.
We define the {\em generalized
  Tchebyshev polynomial $T^L_n(x)$ of the first kind} as the image of
$x^n$ under the unique linear map $T^L:{\mathbb R}[x]\rightarrow {\mathbb
  R}[x]$ that has the following property: given any simplicial complex
$K$ and any
generalized Tchebyshev triangulation $K'$ of $K$, induced by $L$, we
have
\begin{equation}
\label{eq:T1}
T^L(F(K,x))=F(K',x).
\end{equation}
\end{definition}
The linear map $T^L$ in Definition~\ref{def:T1} above is well-defined:
let $T_1:\mathbb{R}[x]\rightarrow \mathbb{R}[x]$ be the invertible linear map satisfying $T_1(f(K,x))=F(K,x)$ for all simplicial complexes $K$, and
$T_2:\mathbb{R}[x]\rightarrow \mathbb{R}[x]$ be the linear map from Theorem~\ref{thm:samefxy} satisfying $T_2(f(K,x))=f(K',x)$ (plugging $y=x$). Then $T^L=T_1T_2T_1^{-1}$.
We now compute $T^L_n(x)$ explicitly.
%%%%%%%%
%we only need to require the validity of \eqref{eq:T1} for the
%case when $K$ is a simplex, the validity for other simplicial complexes
%follows by linearity, just like in the proof of
%Theorem~\ref{thm:samefxy}.
When $K$ is an $(n-1)$-dimensional simplex,
we have
$$
F(K,x)=\sum_{j=0}^n \binom{n}{j} \left(\frac{x-1}{2}\right)^j=
\left(\frac{x+1}{2}\right)^n\quad\mbox{and}\quad
F(K',x)=f_n\left(\frac{x-1}{2},\frac{x-1}{2}\right).
$$
As a consequence $T^L$ is given by
\begin{equation}
\label{eq:TL1a}
T^L\left(\left(\frac{x+1}{2}\right)^n\right)
=f_n\left(\frac{x-1}{2},\frac{x-1}{2}\right).
\end{equation}
Since
$$
x^n=\left(2\cdot \frac{x+1}{2}-1\right)^n
=\sum_{k=0}^n \binom{n}{k}(-1)^{n-k} 2^k \left(\frac{x+1}{2}\right)^k,
$$
 equation (\ref{eq:TL1a}) is equivalent to
\begin{equation}
\label{eq:TL1}
T^L_n(x)=T^L(x^n)=
\sum_{k=0}^n \binom{n}{k}(-1)^{n-k} 2^k
f_k\left(\frac{x-1}{2},\frac{x-1}{2}\right).
\end{equation}
Using \eqref{eq:TL1} we obtain the following generating function formula for the
polynomials $T^L_n(x)$.
\begin{align*}
\sum_{n=0}^{\infty} T^L_n(x) t^n
&=\sum_{k=0}^{\infty}2^k f_k\left(\frac{x-1}{2},\frac{x-1}{2}\right)
 t^k\sum_{m=0}^{\infty}\binom{m+k}{k} (-t)^m\\
&=\sum_{k=0}^{\infty} f_k\left(\frac{x-1}{2},\frac{x-1}{2}\right)
\cdot \frac{(2t)^k}{(1+t)^{k+1}}, \quad\mbox{i.e.,}
\end{align*}
\begin{equation}
\label{eq:Tf}
\sum_{n=0}^{\infty} T^L_n(x) t^n
=
\frac{1}{1+t}
f\left(\frac{x-1}{2},\frac{x-1}{2},\frac{2t}{1+t}\right).
\end{equation}
%%%%%%%%%%%%%%%%%%%%%%%%%%%%%%%%%%%%
%%%%%%%%%%%%%%%%%%%%%%
This equation and Proposition~\ref{prop:fgen} yield
\begin{equation}
\label{eq:TLgen}
\sum_{n=0}^{\infty} T^L_n(x) t^n=
\frac{1}{1-xt} \cdot
\frac{r_L\left(-\frac{1+t}{2t}, -\frac{1+x}{2}\right)-r_L\left(-\frac{1+x}{2},-\frac{1+x}{2}\right)}{r_L\left(-\frac{1+t}{2t}, -\frac{1+x}{2}\right)}.
\end{equation}
Combining Equation \eqref{eq:TL1} with (\ref{eq:fIf}) yields
\begin{align*}
T^L_n(x)&=\sum_{k=0}^n \binom{n}{k}(-1)^{n-k} 2^k
\sum_{j=0}^k \binom{k}{j}\If_j\left(\frac{x-1}{2},\frac{x-1}{2}\right)\\
&=\sum_{j=0}^n
\If_j\left(\frac{x-1}{2},\frac{x-1}{2}\right)2^j\binom{n}{j}
\sum_{k=j}^n \binom{n-j}{k-j} (-1)^{n-k} 2^{k-j}.\\
\end{align*}
The inside sum is $(2-1)^{n-j}=1$ and we obtain
\begin{equation}
T^L_n(x)=\sum_{j=0}^n
\If_j\left(\frac{x-1}{2},\frac{x-1}{2}\right)2^j\binom{n}{j},
\end{equation}
where $2^j\binom{n}{j}$ is the number of $(j-1)$-dimensional faces of an
$n$-dimensional cross-polytope. Thus by (\ref{eq:if}) we observe that:
\begin{corollary}
\label{cor:Tcross}
$T^L_n(x)$ is the $F$-polynomial of the generalized Tchebyshev
  triangulation of the boundary complex of an $n$-dimensional
  cross-polytope, induced by $L$.
\end{corollary}
Corollary~\ref{cor:Tcross} allows us to prove several properties of the
generalized Tchebyshev polynomials of the first kind.
\begin{theorem}
\label{thm:gtch1}
For all $n\geq 0$, the polynomials $T^L_n(x)$ have the following properties:
\begin{enumerate}
\item $T^L_n(x)$ is a polynomial of degree $n$;
\item $T^L_n(1)=1$;
\item $(-1)^n T^L_n(-x)=T^L_n(x)$;
\item all real roots of $T^L_n(x)$ belong to the interval $(-1,1)$.
\end{enumerate}
\end{theorem}
\begin{proof}
Let $(f_{-1},\ldots,f_{n-1})$, respectively $(h_0,\ldots,h_n)$ be the
$f$-vector and $h$-vector, respectively, of the generalized Tchebyshev
  triangulation of the boundary complex of an $n$-dimensional
  cross-polytope, induced by $L$. By Corollary~\ref{cor:Tcross} we have
$$
T^L_n(x)=\sum_{j=0}^n f_{j-1} \left(\frac{x-1}{2}\right)^j.
$$
Clearly $T^L_n(x)$ has degree $n$. Substituting $f_{j-1}=\sum_{i=0}^j
\binom{n-i}{n-j} h_i$ for each $j$, the previous equation may be rewritten as
$$
T^L_n(x)
=\sum_{j=0}^n \left(\frac{x-1}{2}\right)^j \sum_{i=0}^j \binom{n-i}{n-j} h_i
=\sum_{i=0}^n h_i \sum_{j=i}^n \binom{n-i}{n-j} \left(\frac{x-1}{2}\right)^j.
$$
By the binomial theorem we obtain
\begin{equation}
\label{eq:TLh}
T^L_n(x)=\frac{1}{2^n}\sum_{i=0}^n h_i (x-1)^{i}(x+1)^{n-i}.
\end{equation}
Substituting $x=1$ into (\ref{eq:TLh}) yields $T^L_n(1)=h_0=1$.
The third statement follows from the Dehn-Sommerville equations
$h_i=h_{n-i}$.

As a consequence, the set of real zeros of $T^L_n(x)$ is
symmetric to the origin. Thus, to prove the last statement, we only need
to show that $T^L_n(x)$ has no real zero that is larger than $1$. This
is an immediate consequence of (\ref{eq:TLh}) and the fact that the
$h$-vector of a simplicial sphere has only nonnegative entries, with
$h_0=h_n=1$ being strictly positive.
\end{proof}
We remark that the above proof shows that the statements in Theorem \ref{thm:gtch1} are valid for the $F$-polynomial of any homology sphere.

We conclude this section with a recursive description of the polynomials
$T^L_n(x)$.
\begin{theorem}
\label{thm:TLrec}
The polynomials  $T^L_n(x)$ satisfy $T^L_n(x)=x^n$ for $n\leq k$. For
all $n\geq k+1$, the polynomial $T^L_n(x)$ satisfies a recurrence of the
form
$$
T^L_n(x)=\sum_{j=1}^{k+1} p^L_j(x) T^L_{n-j}(x).
$$
Here each $p^L_j(x)$ is a polynomial of $x$ and it equals to the
coefficient of $t^j$ in $(-2t)^{k+1}r_L\left(-\frac{1+t}{2t},
-\frac{1+x}{2}\right)$.
\end{theorem}
\begin{proof}
For $n\leq k$, the generalized Tchebyshev triangulation of the boundary
complex of an $n$-dimensional cross-polytope is the boundary complex
itself whose $F$-polynomial is $x^n$.

To prove the second part of the statement, let us rewrite
(\ref{eq:TLgen}) as
$$
\sum_{n=0}^{\infty} T^L_n(x) t^n=
\frac{1}{1-xt} \cdot
\frac{(-2t)^{k+1}r_L\left(-\frac{1+t}{2t},
  -\frac{1+x}{2}\right)-
  (-2t)^{k+1}r_L\left(-\frac{1+x}{2},-\frac{1+x}{2}\right)}{(-2t)^{k+1}r_L\left(-\frac{1+t}{2t},
  -\frac{1+x}{2}\right)}.
$$
Since the total degree in $u$ and $v$ of each term of $r_L(u,v)$ is at
most $k+1$, the denominator and the numerator of the second factor on
the right hand side are polynomials of $x$ and $t$. Substituting $x=1/t$
into the numerator on the right hand side makes it vanish. As a
consequence, we may always simplify by $(1-tx)$ on the right hand side,
yielding a numerator of degree at most $k$ in $t$. The degree
of the denominator $(-2t)^{k+1}r_L\left(-\frac{1+t}{2t},
-\frac{1+x}{2}\right)$, as a polynomial of $t$ is exactly $k+1$, and the
coefficient of $t^{0}$ is $(-1)$ since $t^{k+1}$ comes only from the
term $-u^{k+1}$ of $r_L(u,v)$. Multiplying both sides with the
denominator on the right hand side and comparing coefficients of $t^n$
on both sides yields a recurrence of the stated form.
\end{proof}
\begin{remark}
\label{rem:notorth}
Theorem~\ref{thm:TLrec} implies that $\{T^L_n(x)\}_{n\geq 0}$ is not a sequence
of orthogonal polynomials if the dimension of $L$ is more than
one. Indeed, every sequence $\{P_n(x)\}_{n\geq 0}$ of monic orthogonal
polynomials satisfies a recurrence of the form
$$P_n(x)=(x-c_n)P_{n-1}(x)-\lambda_n P_{n-2}(x)
\quad\mbox{for $n\geq 1$},$$
where $P_{-1}(x)=0$, $P_0(x)=1$, the numbers $c_n$ and
$\lambda_n$ are constants, $\lambda_n\neq 0$ for $n\geq 2$, and
$\lambda_1$ is arbitrary (see \cite[Ch. I, Theorem 4.1]{Chihara}).
If the dimension of $L$ is greater than one, we have $T^L_n(x)=x^n$ for
$n\leq 2$, forcing $c_1=0$, $c_2=0$ and $\lambda_2=0$; in contradiction
with the requirement of $\lambda_n\neq 0$ for $n\geq 2$.
\end{remark}
Theorem~\ref{thm:TLrec} may be used to find an explicit formula for
$T^L_n(x)$, whenever the characteristic equation associated to the
linear recurrence can be solved. Note that, by Theorem~\ref{thm:TLrec},
this characteristic equation is obtained by replacing each $t^j$ by
$q^{k+1-j}$ in $(-2t)^{k+1}r_L(-(1+t)/2t,-(1+x)/2)$ and finding the
zeros of the resulting polynomial of $q$. This transformation is the
same as substituting $t=1/q$ and multiplying by $q^{k+1}$, thus the
characteristic equation of the linear recurrence is
\begin{equation}
\label{eq:chareq}
(-2)^{k+1}r_L\left(-\frac{1+q}{2},-\frac{1+x}{2}\right)=0.
\end{equation}
If we find $k+1$ linearly independent solutions $q_0(x), q_1(x),\ldots,
q_{k}(x)$ of Equation (\ref{eq:chareq}) above, then we may
look for a general formula of the form
$$
T^L_n(x)=\alpha_0(x)q_0(x)^n+\cdots +\alpha_k(x)q_k(x)^n.
$$
Since $T^L_n(x)=x^n$ holds for $n\leq k$, the array of functions
$(\alpha_0(x),\ldots,\alpha_k(x))$ may be found as the solution of the
system of equations
\begin{equation}
\label{eq:matrixform}
\begin{pmatrix}
1& 1& \cdots & 1\\
q_0(x) & q_1(x) & \cdots & q_k(x)\\
\vdots & \vdots & \ddots& \vdots\\
q_0(x)^k & q_1(x)^k & \cdots & q_k(x)^k\\
\end{pmatrix}
\begin{pmatrix}
\alpha_0(x)\\
\alpha_1(x)\\
\vdots\\
\alpha_k(x)\\
\end{pmatrix}
=
\begin{pmatrix}
1\\
x\\
\vdots\\
x^k\\
\end{pmatrix}.
\end{equation}
Such a system of equations may be solved using Cramer's rule and the
formula for the Vandermonde determinant. Explicit examples will be
worked out in Section~\ref{sec:rr}.

\section{Generalized Tchebyshev polynomials of
  the first kind and real-rootedness}
\label{sec:rr}

By Theorem~\ref{thm:gtch1} the generalized Tchebyshev polynomials of the
first kind $T^L_n(x)$ possess many important properties of the ordinary
Tchebyshev polynomials of the first kind $T_n(x)$. An important property of
the polynomials $T_n(x)$ is that all their roots are
distinct and real. Since $T^L_n(x)=x^n$ holds for $n\leq k$, for $k\geq
2$ the roots of $T^L_n(x)$ are not distinct for all $n$ any more. The
question remains whether all roots of all polynomials $T^L_n(x)$ could
still be real. In this section we explore this question.

We begin with a complete description of the case $k=1$. The only way to
subdivide a $1$-dimensional simplex is to select $s\geq 1$ distinct
vertices in its interior, thus creating a path of length $s+1$.
The magic polynomial $r_L(u,v)$ is given by
$$
r_L(u,v)=sv+2uv+(s-1)v^2-u^2.
$$
To use Theorem~\ref{thm:TLrec}, we observe that
$$
(-2t)^{2}r_L\left(-\frac{1+t}{2t},  -\frac{1+x}{2}\right)
=t^2((x^2-1)(s-1)-1)+2xt-1,
$$
yielding the recurrence
$$
T^L_n(x)=2x\cdot T^L_{n-1}(x)+((x^2-1)(s-1)-1)T^L_{n-2}(x)
\quad\mbox{for $n\geq 2$.}
$$
Note that for $s=1$ the above recurrence degenerates into the well-known
recurrence of the Tchebyshev polynomials $T_n(x)$. Taking into account
the initial conditions $T^L_0(x)=1$ and $T^L_1(x)=x$ it is not hard to
derive (after solving a quadratic characteristic equation) the following
explicit formula:
\begin{equation}
\label{eq:k=1}
T^L_n(x)=\frac{(x-\sqrt{s(x^2-1)})^n+(x+\sqrt{s(x^2-1)})^n}{2}\quad\mbox{for
  $n\geq 0$.}
\end{equation}
\begin{proposition}
\label{prop:Tdim1}
Let $s\geq 1$ be an integer and $L$ be the subdivision of the $1$-simplex by $s$ interior vertices.
Then the polynomial $T^L_n(x)$
%, given in (\ref{eq:k=1}),
has $n$ distinct real roots
in the open interval $(-1,1)$.
\end{proposition}
\begin{proof}
Consider the function
$$
\phi(x)=\frac{x}{\sqrt{x^2+s(1-x^2)}}
$$
on the interval $[-1,1]$. Its derivative, $\phi'(x)=s/(s(1-x^2)+x^2)^{3/2}$, is
positive on $(-1,1)$, thus $\phi(x)$ is strictly increasing on
$(-1,1)$. Obviously we also have $\lim_{x\rightarrow -1} \phi(x)=-1$ and
$\lim_{x\rightarrow 1} \phi(x)=1$. Therefore $\phi(x)$ is a bijection
from $[-1,1]$ to itself. The function $\alpha(x):=\arccos(\phi(x))$ is
well-defined and maps the interval $[-1,1]$ bijectively onto the
interval $[0,\pi]$. Using (\ref{eq:k=1}), it is not difficult to show that we have
\begin{equation}
\label{eq:Ttrig}
T^L_n(x)=\left(\sqrt{x^2+s(1-x^2)}\right)^n \cos(n\alpha(x)),
\end{equation}
which, for $s=1$, is equivalent to the first half of
(\ref{eq:Ttrigdef}).
Now the statement follows from the fact that there are $n$ different
values of $\alpha$ in $(0,\pi)$ for which $\cos(n\alpha)=0$.
\end{proof}

Another interesting special case is when $L$ is obtained by adding just
one vertex to the interior of a $k$-dimensional simplex and we subdivide
the simplex into $k+1$ facets by connecting this new vertex to all other
vertices of the simplex. The resulting magic polynomial is
$$
r_L(u,v)=(1+u)^{k+1}v-(1+v)u^{k+1}.
$$
\begin{example}
When $k=3$ for the complex $L$ above, direct calculation shows that
$$
T^L_6(x)=6-9x^2-60x^4+64x^6.
$$
By Descartes' rule of signs, the polynomial $6-9x-60x^2+64 x^3$ has at
most two positive roots. As a consequence $T^L_6(x)$ has at most
$4$ real roots (Maple finds $4$ real roots indeed, but this is
unimportant). None of these roots can be a double root, because the
derivative of $T^L_6(x)$ is relatively prime to $T^L_6(x)$. Therefore not
all roots of $T^L_6(x)$ are real.
\end{example}

%As a special case of Theorem~\ref{thm:2dimRealRoot}, in the next section
%we will see that, for $k=2$, all roots of $T^L_n(x)$ are real.

This, however, can not happen when $\dim L =2$, as the following theorem shows.
\begin{theorem}
\label{thm:2dimRealRoot}
Let $L$ be any subdivision of the triangle, with no new vertices added
to the boundary. Then the polynomials $T^L_n(x)$ have only real roots.
\end{theorem}
%The proof of this theorem is quite lengthy and technical and is delayed
%to the Appendix.
We conclude this section by explaining why
Theorem~\ref{thm:2dimRealRoot} is a special case of
Theorem~\ref{thm:pn}, which will be stated and shown in the next
section.

Let $m$ be the number of interior vertices in $L$ and let $e$ be the
number of edges in $L$ with one end on the boundary and the other end in
the interior of $L$. Thus the total number  of vertices in $L$ is
$$f_0(L)=m+3.$$
Each edge, except for the three edges on the boundary, is
included in exactly two faces, yielding $2f_1(L)=3(f_2(L)+1)$, whereas Euler's
formula gives $f_0(L)+f_2(L)=f_1(L)+1$. Solving these equations for
$f_1(L)$ and $f_2(L)$ yields
$$
f_1(L)=3(m+1)\quad\mbox{and}\quad f_2(L)=2m+1.
$$
In order to compute the magic polynomial, we need to refine the above
face count. Let us say that a face has type $(i,j)$ if it has $i$
vertices on the boundary and $j$ vertices in the interior.
Of the $3m+3$ edges, $3$ edges have type $(2,0)$, $e$ edges have type
$(1,1)$, and the remaining $3m-e$ edges have type $(0,2)$. Of the $2m+1$
$2$-faces, three have type $(2,1)$. To count the number of faces of type
$(1,2)$, observe that each face of type $(1,2)$ or $(2,1)$ contains
exactly two edges of type $(1,1)$ and, conversely, each edge of type
$(1,1)$ belongs to exactly two faces of type $(1,2)$ or $(2,1)$. Thus
the total number of faces of types $(1,2)$ or $(2,1)$ is the
same as the number of type $(1,1)$ edges, that is, $e$. Since the number
of type $(2,1)$ faces is $3$, there are $e-3$ faces of type
$(1,2)$. Finally, the remaining $2m+1-e$ faces must have type $(0,3)$.
Therefore the magic polynomial associated to $L$ is
\begin{equation}
\label{eq:2dmagic}
r_L(u,v)=mv+euv+(3m-e)v^2+3u^2v+(e-3)uv^2+(2m+1-e)v^3-u^3.
\end{equation}
A closer look at the face-counting argument above also implies the
following statement.
\begin{lemma}
\label{lem:em}
The parameters $m$ and $e$ above satisfy $m\geq 1$ and $0<e\leq 2m+1$.
\end{lemma}
Indeed, $2m+1-e$ is the number of faces of type $(0,3)$ and there is at
least one edge of type $(1,1)$ and at least one vertex added in the
interior.

By (\ref{eq:2dmagic}) we have
$$
(-2t)^{3}r_L\left(-\frac{1+t}{2t}, -\frac{1+x}{2}\right)
=((2m+1-e)\cdot x^3+(e-2m)x)\cdot t^3+((e-3)x^2-e)\cdot t^2+3x\cdot t -1.
$$
By Theorem~\ref{thm:TLrec}, the polynomials $T^L_n(x)$ satisfy the
initial conditions $T^L_0(x)=1$, $T^L_1(x)=x$, $T^L_2(x)=x^2$ and the
recurrence
$$
T^L_n(x)=3xT^L_{n-1}(x)+((e-3)x^2-e)T^L_{n-2}(x)+((2m+1-e)\cdot
x^3+(e-2m)x)\cdot T^L_{n-3}(x)
\quad\mbox{for $n\geq 3$.}
$$

\section{A general real-rootedness result}
\label{sec:pn}

In this section we show a generalization of
Theorem~\ref{thm:2dimRealRoot} that seems interesting by its own
merit. This section may be read independently of the geometric and
combinatorial considerations in the rest of the paper. The only
references we make to a preceding section are reminders of
the end of Section~\ref{sec:TL}, where we recall a well-known way of
solving linear recurrences. We will apply the formulas obtained
using that method.

\begin{theorem}
\label{thm:pn}
Let $m$ and $e$ be positive real numbers satisfying $m\geq 1$ and
$0<e\leq 2m+1$. Assume the sequence $\{p_n(x)\}_{n=0}^{\infty}$ of
polynomials satisfies $p_n(x)=x^n$ for $n\in\{0,1,2\}$ and the recurrence
$$
p_n(x)=3x p_{n-1}(x)+((e-3)x^2-e) p_{n-2}(x)+((2m+1-e)\cdot
x^3+(e-2m)x)\cdot p_{n-3}(x)
\quad\mbox{for $n\geq 3$.}
$$
 Then all roots of all polynomials $p_n(x)$ are real.
\end{theorem}

\begin{remark}
\label{rem:e3m}
Note that an immediate consequence of $m\geq 1$ and $e\leq 2m+1$ is that
we also have $e\leq 3m$ with equality only being possible when $m=1$ and $e=3$.
\end{remark}

The characteristic equation associated to the above recurrence for $p_n(x)$ is
\begin{equation}
\label{eq:char}
(q-x)^3+e(1-x^2)(q-x)+2mx(1-x^2)=0.
\end{equation}
According to Cardano's formula, this characteristic equation has the
following three solutions:
\begin{equation}
\label{eq:qj}
q_j(x)=x+\omega^j u(x)+\omega^{2j} v(x)
\end{equation}
where $j\in\{0,1,2\}$, $\omega=\exp(i\cdot 2\pi/3)$,
\begin{equation}
\label{eq:ux}
u(x)=\sqrt[3]{1-x^2}\cdot
\sqrt[3]{-mx+\sqrt{m^2x^2+\frac{e^3(1-x^2)}{27}}}
\quad\mbox{and}
\end{equation}
\begin{equation}
\label{eq:vx}
v(x)=\sqrt[3]{1-x^2}\cdot \sqrt[3]{-mx-\sqrt{m^2x^2+\frac{e^3(1-x^2)}{27}}}.
\end{equation}
We restrict the domain of the functions $q_j(x)$ to real values of
$x$ in the interval $[-1,1]$. Note that $q_0(x)$ is a real-valued
function, whereas $q_1(x)$ and $q_2(x)$ are complex valued functions
such that $q_2(x)$ is the complex conjugate of $q_1(x)$.
The common length of $q_1(x)$ and $q_2(x)$ is given by
$$
\|q_1(x)\|^2=\|q_2(x)\|^2=q_1(x)\cdot q_2(x)
=(x+\omega u(x)+\omega^2 v(x))(x+\omega^2 u(x)+\omega
v(x)),\quad\mbox{that is,}
$$
\begin{equation}
\label{eq:q1l}
\|q_1(x)\|^2=\|q_2(x)\|^2=x^2-(u(x)+v(x))\cdot x+u(x)^2+v(x)^2-u(x)v(x).
\end{equation}
Similarly, for $j=0$, (\ref{eq:qj}) yields
$|q_0(x)|^2=(x+u(x)+v(x))^2$, that is,
\begin{equation}
\label{eq:q0}
|q_0(x)|^2=x^2+u(x)^2+v(x)^2+2x(u(x)+v(x))+2u(x)v(x).
\end{equation}
For future reference we note that
\begin{equation}
\label{eq:uvx=0}
u(0)=\sqrt{e/3}, \quad
v(0)=-\sqrt{e/3},\quad \mbox{implying}
\end{equation}
\begin{equation}
\label{eq:qjx=0}
q_0(0)=0, \quad
q_1(0)=\sqrt{e} i\quad\mbox{and}\quad
q_2(0)=-\sqrt{e} i.
\end{equation}
Similarly
\begin{equation}
\label{eq:uvx=pm1}
u(1)=u(-1)=v(1)=v(-1)=0\quad\mbox{implies}
\end{equation}
\begin{equation}
\label{eq:qjx=pm1}
q_j(-1)=-1\quad\mbox{and}\quad q_j(1)=1\quad\mbox{for $j=0,1,2$.}
\end{equation}
As a part of the derivation of Cardano's formula, $u(x)$ and $v(x)$ are
known to satisfy the following equalities:
\begin{equation}
\label{eq:p}
u(x)\cdot v(x)=-\frac{e}{3}\cdot (1-x^2),\quad\mbox{and}
\end{equation}
\begin{equation}
\label{eq:q}
u(x)^3+v(x)^3=2mx(x^2-1).
\end{equation}
Besides these classical identities, we will use the following two
inequalities about $u(x)$ and $v(x)$.
\begin{lemma}
\label{lem:u3-v3}
The function $u(x)^3-v(x)^3$ is nonnegative on the interval
$[-1,1]$. Equality to zero holds only when $x=\pm1$.
\end{lemma}
This lemma is a direct consequence of
$$
u(x)^3-v(x)^3=2(1-x^2)\sqrt{m^2x^2+\frac{e^3(1-x^2)}{27}}.
$$
\begin{lemma}
\label{lem:u+v}
The functions $u(x)$ and $v(x)$ satisfy
$$
x(u(x)+v(x))\leq 0
$$
for all real $x\in [-1,1]$. Equality holds exactly when $x=\pm 1$ or $x=0$.
\end{lemma}
\begin{proof}
Consider the function $w: [-1,1]\rightarrow {\mathbb R}$, given by
$$
w(x)=\sqrt[3]{-mx+\sqrt{m^2x^2+\frac{e^3(1-x^2)}{27}}} + \sqrt[3]{-mx-\sqrt{m^2x^2+\frac{e^3(1-x^2)}{27}}}
$$
This is a continuous function on $[-1,1]$, satisfying
$w(1)=\sqrt[3]{-2m}<0$  and $w(-1)=\sqrt[3]{2m}>0$. Furthermore, the
only solution of $w(x)=0$ on
the interval $[-1,1]$ is $x=0$. We obtain that $w(x)$ is positive on
$[-1,0)$ and negative on $(0,1]$. Since, by (\ref{eq:ux}) and
 (\ref{eq:vx}), $w(x)$ satisfies $u(x)+v(x)=\sqrt[3]{1-x^2}\cdot w(x)$,
the sign of $u(x)+v(x)$ is the same as the sign of $w(x)$ for all
$x\in (-1,1)$, and the statement follows directly.
\end{proof}

Just like at the end of Section~\ref{sec:TL}, we may look for
$p_n(x)$ in the form
\begin{equation}
\label{eq:Taq}
p_n(x)=\alpha_0(x)q_0(x)^n+\alpha_1(x)q_1(x)^n+\alpha_2(x)q_2(x)^n,
\end{equation}
where the functions $\alpha_0(x)$, $\alpha_1(x)$ and $\alpha_2(x)$
may be found by solving (\ref{eq:matrixform}).
\begin{lemma}
\label{lem:alpha}
On the interval $(-1,1)$, the  functions $\alpha_0(x)$, $\alpha_1(x)$ and
$\alpha_2(x)$ are given by
$$
\alpha_0(x)=\frac{(u(x)^2-u(x)v(x)+v(x)^2)(u(x)-v(x))}{3(u(x)^3-v(x)^3)},\quad\mbox{and}
$$
$$
\alpha_j(x)=\frac{(u(x)^2+(-1)^{j}i\sqrt{3}u(x)v(x)-v(x)^2)(u(x)+v(x))}{3(u(x)^3-v(x)^3)}\quad\mbox{for
$j=1,2$.}
$$
\end{lemma}
\begin{proof}
We use Cramer's formula to solve (\ref{eq:matrixform}). For all
$\alpha_j(x)$, the denominator in this formula is the Vandermonde determinant
$$
\det\begin{pmatrix}
1& 1& 1\\
q_0(x) & q_1(x) & q_2(x)\\
q_0(x)^2 & q_1(x)^2 & q_2(x)^2\\
\end{pmatrix}
=(q_1(x)-q_0(x))(q_2(x)-q_0(x))(q_2(x)-q_1(x)),
$$
which, by (\ref{eq:qj}), equals
$$
((\omega-1)u(x)+(\omega^2-1)v(x)) ((\omega^2-1)u(x)+(\omega-1)v(x))
((\omega^2-\omega)u(x)+(\omega-\omega^2)v(x)).
$$
After taking out a $(\omega-1)$ from the first factor, $(\omega^2-1)$
from the second factor and $(\omega^2-\omega)$ from the third factor,
and after noting that
$$
(\omega-1)(\omega^2-1)(\omega^2-\omega)=-3\sqrt{3}i,
$$
we obtain that the common denominator in Cramer's formula is
$$
-3\sqrt{3} i(u(x)-\omega^2 v(x))(u(x)-\omega v(x))(u(x)-v(x))
=-3\sqrt{3}i (u(x)^3-v(x)^3).
$$
The numerators in Cramer's  formula are also Vandermonde determinants
and may be computed in a completely analogous way. The stated equalities
follow after simplifying by $-\sqrt{3}i$.
\end{proof}
By Lemma~\ref{lem:u3-v3}, $u(x)^3-v(x)^3$ is real and strictly positive
on the interval $(-1,1)$, hence the formulas stated in
Lemma~\ref{lem:alpha} above are well-defined. In order to
extend the definition of $\alpha_j(x)$ to $x=\pm 1$ in a continuous
fashion, we state the following, equivalent formulas for $\alpha_j(x)$.
\begin{lemma}
\label{lem:alphabis}
On the set $(-1,1)\setminus \{0\}$, the functions $\alpha_j(x)$ are
equivalently given by
\begin{equation}
\label{eq:alphaj}
\alpha_j(x)=\frac{mx}{e(q_j(x)-x)+3mx}\quad\mbox{for $j=0,1,2$.}
\end{equation}
These formulas may be continuously extended to $[-1,1]$
by setting $\alpha_j(1)=1/3$, $\alpha_j(-1)=1/3$ for
$j=0,1,2$, $\alpha_0(0)=1$ and $\alpha_j(0)=0$ for $j=1,2$.
\end{lemma}
\begin{proof}
Observe first that, by (\ref{eq:q}), the sum $u(x)^3+v(x)^3$
is nonzero on the set $(-1,1)\setminus \{0\}$ thus the same holds for
$u(x)+v(x)$ by
$u(x)^3+v(x)^3=(u(x)+v(x))(u(x)^2-u(x)v(x)+v(x)^2)$. Using these
observations, we may rewrite $\alpha_0(x)$ as
$$
\alpha_0(x)=
\frac{\displaystyle \frac{u(x)^3+v(x)^3}{u(x)+v(x)}(u(x)-v(x))}{3(u(x)^3-v(x)^3)}
=
\frac{2mx(x^2-1)}{3(u(x)^2+u(x)v(x)+v(x)^2)(u(x)+v(x))}.
$$
Here $u(x)+v(x)$ may be replaced by $q_0(x)-x$. Furthermore, by (\ref{eq:p}),
the factor $u(x)^2+u(x)v(x)+v(x)^2$ in the denominator above may be
rewritten as
$$
u(x)^2+u(x)v(x)+v(x)^2=(u(x)+v(x))^2-u(x)v(x)=(q_0(x)-x)^2-\frac{e(x^2-1)}{3}.
$$
Thus we obtain
$$
\alpha_0(x)
=\frac{2mx(x^2-1)}{3\left((q_0(x)-x)^2-\frac{e(x^2-1)}{3}\right)(q_0(x)-x)}
=\frac{2mx(x^2-1)}{3((q_0(x)-x)^3-e(x^2-1)(q_0(x)-x))}.
$$
After expanding $(q_0(x)-x)^3$ and using (\ref{eq:char}) to replace
$q_0(x)^3$ with a linear expression of $q_0(x)$, we obtain
$$
\alpha_0(x)=\frac{2mx(x^2-1)}{6mx(x^2-1)+2e(x^2-1)(q_0(x)-x)}.
$$
Simplifying by $2(x^2-1)$ yields the stated equation for $\alpha_0(x)$.
The calculations for $\alpha_1(x)$ and $\alpha_2(x)$ are completely
analogous, therefore omitted.

%In both calculations, the first step is to
%replace the second degree expression of $u(x)$ and $v(x)$ in the
%numerator with $(u(x)^3+v(x)^3)/(u(x)+\omega v(x))$,  respectively
%$(u(x)^3+v(x)^3)/(u(x)+\omega^2 v(x))$, and we use the fact
%that $u(x)^3+v(x)^3$ is not zero on $(-1,1)\setminus \{0\}$.

Substituting $x=1$, respectively $x=-1$, in the stated formulas for
$\alpha_j(x)$
yields $\alpha_j(1)=1/3$ and $\alpha_j(-1)=1/3$, as we have $q_j(1)=1$
and $q_j(-1)=-1$ for $j=0,1,2$. These are obviously continuous
extensions of the functions $\alpha_j(x)$. By (\ref{eq:qjx=0}), for $j\in
\{1,2\}$ the denominator $e(q_j(x)-x)+3mx$ is nonzero at $x=0$ and
$\alpha_j(0)=0$ is a continuous extension of the given formula. Finally,
to find the limit of $\alpha_0(x)$ at $x=0$, observe that using
(\ref{eq:q}) we may rewrite
$$q_0(x)=x+u(x)+v(x)=x+\frac{u(x)^3+v(x)^3}{u(x)^2-u(x)v(x)+v(x)^2}$$
as
$$
q_0(x)=x\left(1+\frac{2m(x^2-1)}{u(x)^2-u(x)v(x)+v(x)^2}\right).
$$
Using (\ref{eq:uvx=0}), the last equation yields
\begin{equation}
\label{eq:q0/x}
\lim_{x\rightarrow 0} \frac{q_0(x)}{x}=\frac{e-2m}{e}.
\end{equation}
Equation (\ref{eq:q0/x}) implies
$$\lim_{x\rightarrow 0} \alpha_0(x)
=\lim_{x\rightarrow 0} \frac{m}{e(q_0(x)/x-1)+3m}=1.
$$
\end{proof}
\begin{definition}
For $j=0,1,2$, we define the functions $\alpha_j(x)$ on the interval $[-1,1]$
by the formulas stated in Lemma~\ref{lem:alphabis}.
\end{definition}
Note that the functions $\alpha_j(x)$ are also given by the
equation (\ref{eq:matrixform}) on the interval $(-1,1)$,
and for such values of $x$ our definition is equivalent to the solution
given in Lemma~\ref{lem:alpha}. Our definition extends these functions to
$x=\pm 1$ in a continuous way, such that they are still solutions of
the system (\ref{eq:matrixform}) which is degenerate for these
values of $x$.

\begin{corollary}
\label{cor:a/x}
The function $\frac{\alpha_1(x)}{x}$ is well-defined and nowhere zero on
$[-1,1]$.
\end{corollary}
Indeed, by Lemma~\ref{lem:alphabis} we may write
\begin{equation}
\label{eq:alpha1/x}
\frac{\alpha_1(x)}{x}
=
\frac{m}{e(q_1(x)-x)+3mx}.
\end{equation}
For a real number $x$, the denominator can only be zero when $q_1(x)-x=\omega
u(x)+\omega^2 v(x)$ is a real number, i.e., when $u(x)=v(x)$. The only
solutions of $u(x)=v(x)$ are $x=\pm 1$. However, by (\ref{eq:qjx=pm1}),
the denominator is nonzero at $x=\pm 1$.

Next we make an analogous observation for $q_1(x)$.

\begin{proposition}
\label{prop:q1neq0}
The function $q_1(x)$ is nowhere zero on the interval $[-1,1]$.
\end{proposition}
\begin{proof}
If $q_1(x)=0$, then (\ref{eq:q1l}) gives
$$
x^2-(u(x)+v(x))\cdot x+u(x)^2+v(x)^2-u(x)v(x)=0.
$$
Consider this as a quadratic equation for $x$, with real coefficients.
It can only have a real solution when its discriminant
$$
D=(u(x)+v(x))^2-4(u(x)^2+v(x)^2-u(x)v(x))
$$
is not negative. Using (\ref{eq:p}), the discriminant may be rewritten as
$$
D=-3(u(x)^2+v(x)^2)+2e(x^2-1).
$$
Here $-3(u(x)^2+v(x)^2)$ is at most zero, and, for $x\in [-1,1]$, we
also have $2e(x^2-1)\leq 0$. Thus $D\geq 0$ is only possible when $x=\pm
1$. However, $q_1(x)$ is not zero  at $x=\pm 1$, as we have  $q_1(1)=1$
and $q_1(-1)=-1$.
\end{proof}

The proof of the main result of this section depends on two key
inequalities, stated in the next two propositions.
\begin{proposition}
\label{prop:q1q0}
We have $\|q_1(x)\|\geq |q_0(x)|$ for all $x\in [-1,1]$. Equality holds
exactly when $x=\pm 1$.
\end{proposition}
\begin{proof}
The difference of (\ref{eq:q1l}) and (\ref{eq:q0}) is
$$
\|q_1(x)\|^2-|q_0(x)|^2=-3x(u(x)+v(x))-3u(x)v(x).
$$
Here, for any $x\in[-1,1]$, the summand $-3x(u(x)+v(x))$
is nonnegative by Lemma~\ref{lem:u+v} and the summand $-3u(x)v(x)$ is
nonnegative by Equation~(\ref{eq:p}). The sum is zero only when both
summands are zero, which is only possible when $x=\pm1$.
\end{proof}
\begin{proposition}
\label{prop:aq}
The functions $\alpha_j(x)$ and $q_j(x)$ satisfy
$$
2\|\alpha_1(x)q_1(x)\| \geq |\alpha_0(x)q_0(x)|
$$
on the interval $[-1,1]$. Equality is only possible when $x=0$.
\end{proposition}
\begin{proof}
Assume, by way of contradiction, that
$$
|\alpha_0(x)q_0(x)|\geq 2
\|\alpha_1(x)q_1(x)\|=\|\alpha_1(x)q_1(x)\|+\|\alpha_2(x)q_2(x)\|$$
holds for some $x\in [-1,1]\setminus \{0\}$. Then, by the triangle inequality, we also
have
$$
|\alpha_0(x)q_0(x)|\geq \|\alpha_1(x)q_1(x)+\alpha_2(x)q_2(x)\|.
$$
Using (\ref{eq:Taq}) with $n=1$ yields
$$
|\alpha_0(x)q_0(x)|\geq |x-\alpha_0(x)q_0(x)|.
$$
Since we excluded the possibility of $x=0$, we obtain that the sign of
$\alpha_0(x)q_0(x)$ must equal to the sign of $x$. Using
(\ref{eq:alphaj}) and the Vi\`ete formulas associated to the
characteristic equation (\ref{eq:char}) it is easy to derive the
following formula:
$$
\alpha_0(x)\alpha_1(x)\alpha_2(x)
=
\frac{m^2x^2}{27m^2 x^2 +e^3(1-x^2)}
$$
On the left hand side, $\alpha_1(x)\alpha_2(x)=\|\alpha_1(x)\|^2$ is
positive by Corollary~\ref{cor:a/x}. The right hand side is also
positive. We obtain that $\alpha_0(x)$ must be positive and thus the sign of
$x$ must also equal the sign of $q_0(x)$. Since we also have $3m-e\geq
0$ (see Remark~\ref{rem:e3m}), using
(\ref{eq:alphaj}) we may write
$$
|\alpha_0(x)q_0(x)|=|mx|\left|\frac{q_0(x)}{eq_0(x)+(3m-e)x}\right|=
|mx|\frac{|q_0(x)|}{e|q_0(x)|+(3m-e)|x|}.
$$
The rightmost expression can only increase if we replace $|q_0(x)|$ with
a larger number. Thus, Proposition~\ref{prop:q1q0} yields
$$
|\alpha_0(x)q_0(x)|\leq |mx|\frac{\|q_1(x)\|}{e\|q_1(x)\|+(3m-e)|x|}.
$$
Applying the triangle inequality to the denominator on the right hand
side yields
$$
|\alpha_0(x)q_0(x)|\leq
\left\|\frac{mxq_1(x)}{eq_1(x)+(3m-e)x}\right\| =\|\alpha_1(x)q_1(x)\|,
$$
which contradicts our assumptions unless
$\alpha_1(x)q_1(x)=\alpha_0(x)q_0(x)=0$, impossible for $x\neq 0$ by
Corollary~\ref{cor:a/x} and Proposition~\ref{prop:q1neq0}.
\end{proof}

As a consequence of Corollary~\ref{cor:a/x} and
Proposition~\ref{prop:q1neq0}, for $n\geq 1$ we may rewrite (\ref{eq:Taq}) as
\begin{equation}
\label{eq:Taq/x}
\frac{p_n(x)}{x}
=\frac{\|\alpha_1(x) q_1(x)^n\|}{|x|}
\left( \frac{\alpha_0(x) \frac{q_0(x)}{x}}{\left\|\frac{\alpha_1(x)}{x} q_1(x)\right\|} \left(\frac{q_0(x)}{\|q_1(x)\|}\right)^{n-1}
+\sum_{j=1}^2
\frac{\frac{\alpha_j(x)}{x}q_j(x)}{\left\|\frac{\alpha_1(x)}{x}
  q_1(x)\right\|}
\left(\frac{q_j(x)}{\|q_1(x)\|}\right)^{n-1}\right).
\end{equation}
Introducing the functions
$$
g_n(x)=\frac{\alpha_0(x) \frac{q_0(x)}{x}}{\left\|\frac{\alpha_1(x)}{x} q_1(x)\right\|}
\left(\frac{q_0(x)}{\|q_1(x)\|}\right)^{n-1},\quad
\varepsilon(x)=\frac{\frac{\alpha_1(x)}{x}q_1(x)}{\|\frac{\alpha_1(x)}{x}
  q_1(x)\|}\quad\mbox{and}
\quad \rho(x)=\frac{q_1(x)}{\|q_1(x)\|},
$$
we may rewrite (\ref{eq:Taq/x}) as
\begin{equation}
\label{eq:Taq/x-short}
\frac{p_n(x)}{x}
=\frac{\|\alpha_1(x) q_1(x)^n\|}{|x|}
\left( g_n(x)+\varepsilon(x)\cdot
\rho(x)^{n-1}+\overline{\varepsilon(x)}\cdot \overline{\rho(x)}^{n-1}\right).
\end{equation}
%Summarizing (\ref{eq:qjx=0}), (\ref{eq:qjx=pm1}),
%Lemma~\ref{lem:alphabis}, and Propositions~\ref{prop:q1q0} and
%\ref{prop:aq} we state the following properties of the function $g_n(x)$.
The next three lemmas gather properties of the functions $g_n(x),\varepsilon(x),\rho(x)$ that will be needed later for the proof of real-rootedness.
\begin{lemma}
\label{lem:gn}
For $n>1$, the function $g_n:[-1,1]\rightarrow {\mathbb R}$ is a real-valued
function satisfying $g_n(-1)=(-1)^{n-1}$, $g_n(0)=0$
%$g_n(0)=|(e-2m)/m|$
and
$g_n(1)=1$. Furthermore, there exists a positive constant $c<2$ such that
$|g_n(x)|\leq c$ holds for all $x\in [-1,1]$.
\end{lemma}
\begin{proof}
The function $g_n(x)$ is continuous and real-valued, because the same holds for the functions $\alpha_0(x)$ and $q_0(x)/x$; see (\ref{eq:qj}) and Lemma \ref{lem:alphabis}.
Direct substitution (in Equations (\ref{eq:qjx=0}) and (\ref{eq:qjx=pm1}), using Lemma~\ref{lem:alphabis} and Equations
(\ref{eq:q0/x}) and (\ref{eq:alpha1/x}) ) yields
$g_n(0)=0$, $g_n(1)=1$ and $g_n(-1)=(-1)^{n-1}$.
For $x\neq 0$, we have
$$
|g_n(x)|=\frac{\alpha_0(x)q_0(x)}{\left\|\alpha_1(x)q_1(x)\right\|}
\cdot \left(\frac{|q_0(x)|}{\|q_1(x)\|}\right)^{n-1}
$$
and the inequality is a direct consequence of Propositions~\ref{prop:q1q0} and
\ref{prop:aq} as $|g_n(0)|<2$ (using compactness of $[-1,1]$).
%By continuity, the same inequality may be extended to
%$x=0$.
%$$
%g_n(0)=
%\lim_{x\rightarrow 0}
%\frac{ \left|\alpha_0(x)
%  \frac{q_0(x)}{x}\right|}{\left|\left|\frac{\alpha_1(x)}{x}
%  q_1(x)\right|\right|}
%=
%\frac{1\cdot \left|\frac{e-2m}{e}\right|}{\frac{m}{e\sqrt{e}i} \sqrt{e}i}
%=
%\left|\frac{e-2m}{m}\right|
%$$
%By Lemma~\ref{lem:em} we have $e\leq 2m+1$. If also $e\leq 2m$ holds,
%then we have $g_n(0)=(2m-e)/m$ which is at most $1$ by $e\leq 3m$. If
%$e=2m+1$ then we have $g_n(0)=1/m\leq 1$.
\end{proof}
\begin{lemma}
\label{lem:rho}
The function $\rho: [-1,1]\rightarrow {\mathbb C}$ is a continuous
function whose range is the upper half of the unit circle, centered at
the origin. $\rho(x)$ is real if and only if  $x=\pm 1$, where we have
$\rho(-1)=-1$ and $\rho(1)=1$.
\end{lemma}
\begin{proof}
Clearly $\rho$ is continuous and we must have $\|\rho(x)\|=1$ for all
$x\in [-1,1]$. The imaginary part of $q_1(x)$ is $\sqrt{3}\cdot
(u(x)-v(x))\cdot i$ and $u(x)-v(x)$ is strictly positive on $(-1,1)$,
see (\ref{eq:ux}) and (\ref{eq:vx}).
\end{proof}
\begin{lemma}
\label{lem:epsilon}
The function $\varepsilon : [-1,1]\rightarrow {\mathbb C}$ is continuous
and its range is a proper subset of the unit circle, centered at the
origin. The real number $-1$ is not part of the range. If $e=3m$, then
$\varepsilon(x)=1$ for all $x\in [-1,1]$. If $e\neq
3m$, then $\varepsilon(x)$ is real only when $x\in \{-1,0,1\}$ and, for all
other values of $x$, the  sign of the imaginary part of $\varepsilon(x)$
is the same as the sign of $x$.
\end{lemma}
\begin{proof}
Clearly $\varepsilon$ is continuous and satisfies
$\|\varepsilon(x)\|=1$. Direct substitution (into (\ref{eq:qjx=0}),
(\ref{eq:qjx=pm1})  and (\ref{eq:alpha1/x})) yields $\varepsilon(-1)=1$,
$\varepsilon(1)=1$ and $\varepsilon(0)=1$. In the case when $e=3m$,
we have
$$
\frac{\alpha_1(x)}{x}\cdot q_1(x) =\frac{m}{e q_1(x)}\cdot q_1(x)=\frac{1}{3}
$$
and $\varepsilon$ is identically $1$.  Assume from now on that $e\neq
3m$. Assume also that $x\not\in \{-1,0,1\}$ and
$\varepsilon(x)$ is real.
Substituting (\ref{eq:alphaj})
into the definition of $\varepsilon(x)$ we obtain
$$
\frac{mq_1(x)}{e(q_1(x)-x)+3mx}=r
$$
for some $r\in {\mathbb R}$,
which may be rearranged as
$$
(m-er)q_1(x)=r(3m-e)x.
$$
On the right hand side we have a real number, whereas on the left hand
side $m-er$ is real but $q_1(x)$ is not real for $x\in (-1,1)\setminus
\{0\}$. The two sides can only be equal, if $m-er=0$ but then $x$ must
be zero, in contradiction with our assumptions.

Assume $x\in (0,1)$. We have seen in the proof of Lemma~\ref{lem:rho} that
the imaginary part of $q_1(x)$ is positive.
Since $3m-e$ is positive, the argument of $e\cdot q_1(x)+(3m-e)x$ is
smaller than the argument of $q_1(x)$, but the imaginary part of
$e\cdot q_1(x)+(3m-e)x$ is also positive. We obtain that the argument of
the quotient
$$\frac{\alpha_1(x)}{x}\cdot q_1(x)=\frac{mq_1(x)}{e\cdot
  q_1(x)+(3m-e)x}$$
belongs to the interval $(0,\pi)$ and the imaginary part of
$\varepsilon(x)$ is positive. A completely analogous reasoning may be
used to prove that the imaginary part $\varepsilon(x)$ is negative for
negative $x$.
\end{proof}

%\begin{theorem}
%\label{thm:2dimRealRoot}
%Let $L$ be any subdivision of the triangle, with no new vertices added
%to the boundary. Then the polynomials $T^L_n(x)$ have only real roots.
%\end{theorem}
\begin{proof}[Proof of Theorem \ref{thm:pn}.]
We only need to show the statement for $n\geq 3$. Since we have
$p_n(0)=0$, it suffices to show that the polynomial $p_n(x)/x$ has
$n-1$ distinct roots in the interval $[-1,1]$. Consider the expression
of $p_n(x)/x$ given in (\ref{eq:Taq/x-short}). It suffices to show
that the function
$$
g_n(x)
+\varepsilon(x)\cdot \rho(x)^{n-1}
+\overline{\varepsilon(x)}\cdot \overline{\rho(x)}^{n-1}
$$
has at least $n-1$ zeroes in the interval $[-1,1]$. By
Lemma~\ref{lem:gn}, the graph of the continuous function $-g_n(x)$
is in between the horizontal lines $y=-c$ and $y=c$ for some $0<c<2$. %By Lemmas~\ref{lem:rho} and \ref{lem:epsilon}
As
 $\varepsilon(x)\cdot
\rho(x)^{n-1}$ is a unit complex number,
$$f_n(x):=\varepsilon(x)\cdot
\rho(x)^{n-1}+\overline{\varepsilon(x)}\cdot \overline{\rho(x)}^{n-1}$$
equals twice
the cosine of the argument of $\varepsilon(x)\cdot \rho(x)^{n-1}$.
As a consequence, the graph of the continuous real-valued function
$f_n(x)$ is between the horizontal lines $y=-2$ and $y=2$. At the
endpoints of the interval $[-1,1]$ we have $f(-1)=2\cdot (-1)^{n-1}$ and
$f(1)=2$. It suffices to prove that there are $n-2$ real numbers
$x_1,x_2,\ldots,x_{n-2}$ satisfying $-1<x_1<\cdots<x_{n-2}<1$ and
$f(x_j)=2\cdot (-1)^{n-1-j}$ for $j=1,\ldots,n-2$. Introducing $x_0=-1$
and $x_{n-1}=1$ we can then say that, for each $j\in \{1,\ldots, n-1\}$,
  in each interval $(x_{j-1},x_j)$,  the graph of $f_n(x)$ enters and
leaves the region between $y=-c$ and $y=c$, and crosses the graph of
$-g_n(x)$ at least once, where we have a root of  $f_n(x)+g_n(x)$.

Consider first the special case when $e=3m$. By Lemma~\ref{lem:epsilon}
$\varepsilon(x)$ is identically $1$ and
$\varepsilon(x)\rho(x)^{n-1}=\rho(x)^{n-1}$. By Lemma~\ref{lem:rho}, as
$x$ moves from $-1$ to $1$, the argument of $\rho$ continuously changes
from $\pi$ to $0$. We may select $x_j$ as the least real number for
which the argument of $\rho(x_j)$ is $\frac{n-1-j}{n-1}\pi$. Then the
argument of $\rho(x)^{n-1}$ is  $(n-1-j)\pi$ and we have $f(x_j)=2\cdot
(-1)^{n-1-j}$. Because of the continuity of $\rho$ we must also have
$-1<x_1<\cdots<x_{n-2}<1$.

Consider finally the case when $e\neq 3m$. For $j=0,\ldots,n-1$, let
$z_j$ be the least real number such that the argument of $\rho(z_j)$
is $\frac{n-1-j}{n-1}\pi$. Clearly we have $-1=z_0<z_1<\ldots<z_{n-1}\leq 1$.
Let us set $x_0=-1$ and $x_{n-1}=1$ . Let us denote by $k$ the index for
which we have $z_k<0\leq z_{k+1}$. For $j= 1,\ldots, k$ we will
show that we may select $x_j$ as an element of the interval
$(z_{j-1},z_j)$  and for $j=k+1,\ldots,n-2$ we will show that we may
select $x_j$ as an element of the interval $(z_j,z_{j+1})$. Since this
selection automatically guarantees
$-1=x_0<x_1<\ldots<x_{n-2}<x_{n-1}=1$, we only need to show that the
argument of $\varepsilon(x_j)\rho(x_j)^{n-1}$ is $(n-1-j)\pi$ for the $x_j$
we selected.

{\bf Case 1:} $1\leq j\leq k$, implying $z_j < 0$. By
Lemma~\ref{lem:epsilon}, the imaginary part of
$\varepsilon(x)$ is negative for all $x\in (z_{j-1},z_j)$, in other words,
  the argument of $\varepsilon(x)$ belongs to the interval
  $(-\pi,0)$ and the argument of $\varepsilon(x)^{-1}$ belongs to the
  interval $(0,\pi)$. The graph of the function
  $(n-1-j)\pi+\arg(\varepsilon(x)^{-1}))$ stays strictly between the
  horizontal lines $y=(n-1-j)\pi$ and $y=(n-j)\pi$. As $x$ moves from $z_{j-1}$
  to $z_j$, the argument of $\rho(x)^{n-1}$ moves from $(n-1-j+1)\pi$
  down to $(n-1-j)\pi$, in a continuous fashion. Thus the graph of
  $\arg(\rho(x)^{n-1})$ crosses the graph of
  $(n-1-j)\pi+\arg(\varepsilon(x)^{-1})$ at some $x_j\in
  (z_{j-1},z_j)$. For this $x_j$, the argument of
  $\varepsilon(x_j)\rho(x_j)^{n-1}$ is $(n-1-j)\pi$.

{\bf Case 2:} $k+1\leq j\leq n-2$, implying $z_j\geq 0$. By
Lemma~\ref{lem:epsilon}, the imaginary part of
$\varepsilon(x)$ is positive for all $x\in (z_{j},z_{j+1})$. The handling of
  this case is left to the reader as it is completely analogous to the
  previous case.
\end{proof}

\section{Generalized Tchebyshev polynomials of the higher kind}
\label{sec:higher}

As a direct generalization of the construction introduced
in~\cite{Hetyei-Tch}, we may introduce generalized Tchebyshev
polynomials of the higher kind as follows.
\begin{definition}
\label{def:Tj}
%Let $L$ be a triangulation of the $k$-dimensional
%simplex such that the only vertices of $L$ on the boundary $\partial(L)$
%are the original vertices of the simplex, and
For
$j\in\{2,\ldots,k+1\}$, let us define $U^{L,j}:{\mathbb R}[x]\rightarrow
{\mathbb R}[x]$
as the unique linear map satisfying $U^{L,j}(x^n)=0$ for $n\leq j-2$ and having
the following property: given any simplicial complex
$K$ and any generalized Tchebyshev triangulation $K'$ of $K$, induced by $L$, we
have
\begin{equation}
\label{eq:Tja}
U^{L,j} (F(K,x))=\sum_{\sigma\in K, |\sigma|=j-1} F(\link_{K'}(\sigma),x).
\end{equation}
We define the {\em generalized Tchebyshev polynomial $U^{L,j}_n(x)$ of
  the $j$th kind} by
$$
U^{L,j}_n(x)=2^{1-j}\cdot (j-1)! U^{L,j}(x^{n+j-1}).
$$
\end{definition}
Similarly to the map $T^L$, the linear maps $U^{L,j}$ are well-defined,
as a consequence of Theorem~\ref{thm:samefxy}.
To see this, it is enough to show that
$\sum_{\sigma\in K, |\sigma|=j-1} f(\link_{K'}(\sigma),x)$ depends linearly on $f(K,x)$. By Theorem~\ref{thm:samefxy}, there are linear functionals $l_{i,p}$ such that $f_c(K';x,y)=\sum_{i,p}l_{i,p}(f(K))x^iy^p$. Now, any face $\tau\in K'$ with $|V(K)\cap \tau|=i+j-1$ and $|(V(K')\setminus V(K))\cap \tau|=p$
contributes $\binom{i+j-1}{j-1}$ to the coefficient of $x^{i}y^p$ in the polynomial $\sum_{\sigma\in K, |\sigma|=j-1} f_c(\link_{K'}(\sigma);x,y)$. Thus,
$$\sum_{\sigma\in K, |\sigma|=j-1} f(\link_{K'}(\sigma),z)
=\frac{1}{(j-1)!}\left. \frac{\partial ^{j-1}}{\partial x^{j-1}}
\left(\sum_{i,p}l_{i,p}(f(K))x^iy^p\right)\right|_{\substack{x=y=z}}.
$$

\begin{example}
Let $L$ be the path with two edges, considered in
Examples~\ref{ex:2triang} and \ref{ex:Tch}. Using Theorem~\ref{thm:samefxy}, as an immediate
generalization of~\cite[Proposition 4.4]{Hetyei-Tch} we obtain that the
polynomials $U^{L,2}_n(x)$ are the ordinary Tchebyshev polynomials of the second
kind.
\end{example}
In general, to compute $U^{L,j}$, by linearity it suffices to find its
value when $K$ is an $(n-1)$-dimensional simplex, where $n\geq
j-1$. When $K$ is an $(n-1)$-dimensional simplex, we have
$$
F(K,z)=\left(\frac{z+1}{2}\right)^n\quad\mbox{and}
$$
$$
\sum_{\sigma\in K, |\sigma|=j-1} F(\link_{K'}(\sigma),z)
=\frac{1}{(j-1)!}\left. \frac{\partial ^{j-1}}{\partial x^{j-1}} f_n
\left(x,y\right)\right|_{\substack{x=(z-1)/2\\ y=(z-1)/2}}.
$$
As a consequence,  $2^{1-j}(j-1)! U^{L,j}$ is given by
\begin{equation}
\label{eq:ULja}
2^{1-j}(j-1)! U^{L,j}\left(\left(\frac{z+1}{2}\right)^n\right)
=2^{1-j} \left. \frac{\partial ^{j-1}}{\partial x^{j-1}}
f_n\left(x,y\right)\right|_{\substack{x=(z-1)/2\\ y=(z-1)/2}}.
\end{equation}
Since
$$
z^{n+j-1}=\left(2\cdot \frac{z+1}{2}-1\right)^{n+j-1}
=\sum_{k=0}^{n+j-1} \binom{n+j-1}{k}(-1)^{n+j-1-k} 2^k
\left(\frac{z+1}{2}\right)^k,
$$
Equation \eqref{eq:ULja} is equivalent to
\begin{equation}
\label{eq:ULj}
U^{L,j}_n(z)=
\sum_{k=j-1}^{n+j-1} \binom{n+j-1}{k}(-1)^{n+j-1-k} 2^{1-j+k}
\left. \frac{\partial ^{j-1}}{\partial x^{j-1}}
f_k\left(x,y\right)\right|_{\substack{x=(z-1)/2\\ y=(z-1)/2}}.
\end{equation}
In analogy to the derivation of (\ref{eq:Tf}), we may use \eqref{eq:ULj}
to obtain the following generating function formula for the
polynomials $U^{L,j}_n(x)$.
\begin{equation}
\label{eq:TUf}
\sum_{n=0}^{\infty} U^{L,j}_n(z) t^n
=
\frac{2^{1-j}}{1+t}
\left. \frac{\partial^{j-1}}{\partial x^{j-1}}
f\left(x,y,\frac{2t}{1+t}\right)\right|_{\substack{x=(z-1)/2\\ y=(z-1)/2}}.
\end{equation}
In analogy to Corollary~\ref{cor:Tcross}, a completely analogous
computation has the following consequence.
\begin{corollary}
\label{cor:Ucross}
Let $\diamondsuit(n+j-1)$ be the boundary complex of
an $(n+j-1)$-dimensional cross-polytope, and let $\diamondsuit(n+j-1)'$ be
a Tchebyshev triangulation of it, induced by $L$. Then we have
$$
U^{L,j}_n(x)=2^{1-j}\cdot (j-1)! \sum_{\sigma\in \diamondsuit(n+j-1),
  |\sigma|=j-1} F(\link_{\diamondsuit(n+j-1)'}(\sigma),x).
$$
\end{corollary}
Using this corollary, it is easy to prove the following analogue of
Theorem~\ref{thm:gtch1}.

\begin{theorem}
\label{thm:gtch2}
For all $n\geq 0$, the polynomials $U^{L,j}_n(x)$ have the following properties:
\begin{enumerate}
\item $U^{L,j}_n(x)$ is a polynomial of degree $n$;
\item $(-1)^n U^{L,j}_n(-x)=U^{L,j}_n(x)$;
\item all real roots of $U^{L,j}_n(x)$ belong to the interval $[-1,1]$.
\end{enumerate}
\end{theorem}
Theorem~\ref{thm:gtch2} naturally inspires the question: which
triangulations $L$ induce Tchebyshev polynomials of the higher kind having
only real roots? We postpone the study of this question to a future
occasion. Here we only wish to highlight one important observation that may
help handle this problem in complete analogy of the same question for
the generalized Tchebyshev polynomials of the first kind: as it is the
case for the ordinary Tchebyshev polynomials,  the polynomials
$U^{L,j}_n(x)$ satisfy the same recurrence as the polynomials $T^L_n(x)$.
\begin{theorem}
\label{thm:ULrec}
For all $n\geq k+1$, the polynomials $U^{L,j}_n(x)$ satisfy a recurrence of the
form
$$
U^{L,j}_n(x)=\sum_{\ell=1}^{k+1} p^L_{\ell}(x) U^{L,j}_{n-\ell}(x).
$$
Here each $p^L_{\ell}(x)$ is a polynomial of $x$ and it equals to the
coefficient of $t^{\ell}$ in $(-2t)^{k+1}r_L\left(-\frac{1+t}{2t},
-\frac{1+x}{2}\right)$.
\end{theorem}
\begin{proof}
To obtain a proof of this statement, observe that the proof of
Theorem~\ref{thm:TLrec} depends on (\ref{eq:TLgen}), which follows from
(\ref{eq:Tf}) and from Proposition~\ref{prop:fgen}. In the proof of
Theorem~\ref{thm:TLrec} we observed that on the right hand side of
(\ref{eq:TLgen}) we may simplify
by $(1-tx)$. Note that we can make an analogous observation ``one step
earlier'' about the right hand side of Proposition~\ref{prop:fgen}:
using $r_L(-1/t, -1-y)$ as the common denominator on the right hand
side, we may simplify the numerator  $r_L(-1/t, -1-y)- r_L(-1-x,-1-y)$
by $1-t(x+1)$ and obtain a formula of the form
$$
f(x,y,t)=\frac{\widetilde{r}_L(x,y,t)}{r_L(-1/t, -1-y)}
$$
for some function $\widetilde{r}_L(x,y,t)$ that is a polynomial of $x$,
$y$ and $1/t$. The denominator $r_L(-1/t, -1-y)$ is independent of $x$,
and remains unchanged when we take the partial derivative with respect
to $x$, even repeatedly. We conclude our proof by referring to
(\ref{eq:TUf}) instead of (\ref{eq:Tf}).
\end{proof}
Using Corollary~\ref{cor:Ucross} and Theorem~\ref{thm:ULrec} it is easy
to answer the question on real roots when the dimension of $L$ is $1$.
\begin{proposition}
\label{prop:Udim1}
Let $s\geq 1$ be an integer and $L$ be the subdivision of the
$1$-simplex by $s$ interior vertices. Then the polynomial $U^{L,2}_n(x)$
has $n$ distinct real roots in the open interval $(-1,1)$.
\end{proposition}
\begin{proof}
Using Corollary~\ref{cor:Ucross} we obtain that
$U^{L,2}_0(x)=2^{1-2}\cdot 2=1$ and
$$U^{L,2}_1(x)=2^{1-2}\cdot 4\cdot (1+2\cdot (x-1)/2)=2x.$$
In analogy of (\ref{eq:k=1}) it is easy to derive
\begin{equation}
\label{eq:k=1U}
U^{L,2}_n(x)=\frac{(x+\sqrt{s(x^2-1)})^{n+1}-(x-\sqrt{s(x^2-1)})^{n+1}}
{2\sqrt{s(x^2-1)}}
\quad\mbox{for  $n\geq 0$.}
\end{equation}
The statement now follows from the fact that, in analogy to
(\ref{eq:Ttrig}), we have
\begin{equation}
\label{eq:Utrig}
U^{L,2}_n(x)=\left(\sqrt{x^2+s(1-x^2)}\right)^n
\frac{\sin((n+1)\alpha(x))}{\sin(\alpha(x))},
\end{equation}
where $\alpha(x)$ is the function introduced in the proof of
Proposition~\ref{prop:Tdim1}, and from the observation that
there are $n$ different
values of $\alpha$ in $(0,\pi)$ for which $\sin((n+1)\alpha)=0$.
Note that, for $s=1$, (\ref{eq:Utrig}) is equivalent to the second half of
(\ref{eq:Ttrigdef}).
\end{proof}

%%%%%%%%%%%%%%%%%%%
%% Does the following fit ??
%%%%%%%%%%%%%%%%%%%
\section{Generalized lower bounds on face numbers}
\label{sec:glbf}

We follow \cite{Nevo-Missing}, with notational change that dimension $d$ there is replaced by $d-1$ here.
For $d,i\geq 1$ integers, let $\mathcal{HS}(i,d)$ be the family of $(d-1)$-dimensional homology spheres without missing faces of dimension $>i$. For $\Delta\in \mathcal{HS}(i,d)$ let $g^{(i)}(\Delta):=g^{(d,i)}(h(\Delta,t))$ be the vector of coefficients when expressing the $h$-polynomial $h(\Delta,t)$ in the basis
$B_{d,i}:=(P_{d,i}(t), tP_{d-2,i}(t),
t^2P_{d-4,i}(t),\ldots,t^{\lfloor\frac{d}{2}\rfloor}
P_{d-2\lfloor\frac{d}{2}\rfloor,i}(t)),$
where
$P_{d,i}(t):=(1+t+\cdots +t^i)^q(1+t+\cdots +t^r)$,
and $q\geq 0, 1\leq r\leq i$ are the unique integers such that $d=qi+r$.

\begin{conjecture}\cite[Conjecture 1.5]{Nevo-Missing}\label{conj:missing-g-conj}
If $\Delta\in \mathcal{HS}(i,d)$, then $g^{(i)}(\Delta)\geq 0$
(component-wise).
\end{conjecture}
The case $i\geq d$ gives the usual $g$-vector and the well known $g$-conjecture,
see e.g. \cite{Stanley-greenbook} for more details on the latter,
and the case $i=1$ gives Gal's $\gamma$-vector and conjecture \cite{Gal}.
%The cases $i=d-1,d-2$ are basically known whenever $g\geq 0$ is known, notably for simplicial polytopes.
Generalizing the usual $g$-polynomial and Gal's $\gamma$-polynomial we
introduce the {\em generalized $g$-polynomial}
\begin{equation}
\label{eq:gi}
g^{(i)}(\Delta, t)=\sum_{j=0}^{\lfloor\frac{d}{2}\rfloor}
g^{(i)}(\Delta)_j t^j
\end{equation}
Conjecture~\ref{conj:missing-g-conj} is obviously equivalent to stating
that, for any $\Delta\in \mathcal{HS}(i,d)$, all coefficients in
$g^{(i)}(\Delta, t)$ are nonnegative.

The following related results \cite[Propositions 1.6 and 4.1]{Nevo-Missing} will be needed.
\begin{lemma}\label{lem:missingResults}
Let $\Delta\in \mathcal{HS}(i,d)$ and $\Delta'\in \mathcal{HS}(i,d')$.
\\
(1)
If $g^{(i)}(\Delta)\geq 0$, then $g^{(i+1)}(\Delta)\geq 0$.
\\
(2)
$\Delta * \Delta' \in \mathcal{HS}(i,d+d')$ and
if $g^{(i)}(\Delta)\geq 0$ and $g^{(i)}(\Delta')\geq 0$ then
$g^{(i)}(\Delta * \Delta')\geq 0$.
\end{lemma}

We will verify Conjecture \ref{conj:missing-g-conj} for simplicial spheres arising as generalized Tchebyshev triangulations of $\diamondsuit(d)$, the boundary complex of the $d$-cross polytope, induced by $L$,
for certain triangulations $L$ considered in previous sections.
%which played a role in Section \ref{sec:TL}, for certain triangulations $L$.
Denote any simplicial sphere obtained in this way
by $\diamondsuit(d,L)$.
We have seen in Theorem \ref{thm:samef} that, although the $\diamondsuit(d,L)$ have different combinatorial types, they all have the same $f$-vector, and hence the same generalized $g$-polynomial.

\begin{theorem}\label{thm:stellar}
Let $\Delta\in \mathcal{HS}(i,d)$ and $F\in \Delta$ of dimension $\leq i$. Note that $\link_{\Delta}(F)\in \mathcal{HS}(i,d-|F|)$ and the stellar subdivision $\Delta(F):=\operatorname{Stellar}_{\Delta}(F) \in \mathcal{HS}(i,d)$.
Assume $g^{(i)}(\Delta)\geq 0$ and $g^{(i)}(\link_{\Delta}(F))\geq 0$. Then $g^{(i)}(\Delta(F))\geq 0$.
\end{theorem}

\begin{proof}
Indeed $\link_{\Delta}(F)\in \mathcal{HS}(i,d-|F|)$, see e.g. \cite[Lemma 2.3]{Nevo-Missing}. To see that $\Delta(F) \in \mathcal{HS}(i,d)$ note that the missing faces of $\Delta(F)$ and not of $\Delta$ are $F$ and some missing edges containing the new vertex $v_F$ of $\Delta(F)$.

Let $u\in F$, then the \emph{link condition}  $\link_{\Delta(F)}(uv_F)=\link_{\Delta(F)}(v_F)\cap \link_{\Delta(F)}(u)$ holds. Moreover, this complex is in $\mathcal{HS}(i,d-2)$ and equals the join $\partial(F\setminus u) * \link_{\Delta}(F)$ of two complexes,
where $\link_{\Delta}(F)\in \mathcal{HS}(i,d-|F|)$ and
$\partial(F\setminus u)\in \mathcal{HS}(i,|F|-3)$.
Note that $g^{(|F|-2)}(\partial(F\setminus u),t)= g(\partial(F\setminus u),t) =1$, thus by Lemma~\ref{lem:missingResults}(1) $g^{(i)}(\partial(F\setminus u))\geq 0$.
By assumption, $g^{(i)}(\link_{\Delta}(F))\geq 0$, so by Lemma~\ref{lem:missingResults}(2),
$g^{(i)}(\link_{\Delta(F)}(uv_F))\geq 0$.

Now, the contraction $v_F\mapsto u$ in $\Delta(F)$ results in $\Delta$. An easy computation shows
$h(\Delta(F),t)=h(\Delta,t)+ th(\link_{\Delta(F)}(uv_F),t)$.
Thus the generalized $g$-polynomial satisfies
$$g^{(i)}(\Delta(F),t)=g^{(i)}(\Delta,t)+t
g^{(i)}(\link_{\Delta(F)}(uv_F),t).$$
By our assumption, both summands on the right hand side have nonnegative
coefficients, therefore the same holds for the left hand side.
\end{proof}

\begin{corollary}\label{cor:L=stellar}
Let $L$ be the subdivision of the $j$-simplex with one interior vertex, namely the one obtained by starring. Assume $1\leq j\leq i$.
Then $g^{(i)}(\diamondsuit(d,L))\geq 0$.
\end{corollary}

\begin{proof}
Note that $g^{(1)}(\diamondsuit(d),t)= \gamma(\diamondsuit(d),t)= 1$,
thus by Lemma~\ref{lem:missingResults}(1) $g^{(i)}(\diamondsuit(d))\geq
0$. Now $\diamondsuit(d,L)$ is obtained from $\diamondsuit(d)$ by a
sequence of stellar subdivisions at faces of dimension $j$ in
$\diamondsuit(d)$. As $j\leq i$, thanks to Theorem \ref{thm:stellar}, it
is enough to verify that when subdividing the $(k+1)$'th $j$-face
$F_{k+1}$ of $\Delta=\Delta_0=\diamondsuit(d)$, considered as a face in the $k$th
complex $\Delta_k$, we have $g^{(i)}(\link_{\Delta_k}(F_{k+1}))\geq 0$.

The key observation here is that for any simplicial complex $K$, the operations link and stellar subdivision commute, more precisely,
%if $F,T\in K$ are incomparable w.r.t. inclusion,
if $T\nsubseteq F$ are sets,
then
$$\link_{K(T)}(F)\cong (\link_K(F))(T\setminus F).$$
(Here, if $F'\notin K'$ for a complex $K'$ and a set $F'$ then define $K'(F'):=K'$. For $|T\setminus F|=1$ and $F\cup T\in K$ the isomorphism is given by mapping $v_T$ to the vertex of $T\setminus F$ and the other vertices to themselves.)
Let the order of the $j$-faces in $\Delta$ be $F_1,F_2,\ldots$.
We will prove the following stronger assertion by double induction on $d$ and $k$ (for fixed $1\leq j \leq i$):

(**) If $F\in \Delta$ is of dimension $\geq j$, and does not contain any of $F_1,\ldots,F_k$, then $g^{(i)}(\link_{\Delta_k}(F))\geq 0$.

The base case $d<j$ is trivial and the base case $k=0$ follows for any $d$ as $\link_{\diamondsuit(d)}(F)\cong \diamondsuit(d-|F|)$, so it has $g^{(1)}(\link_{\diamondsuit(d)}(F),t)=1$, hence $g^{(i)}(\link_{\diamondsuit(d)}(F))\geq 0$.
For $k>0$,
$\link_{\Delta_k}(F)\cong (\link_{\Delta_{k-1}}(F))(F_{k}\setminus F)$.
If $F_k\cup F \notin \Delta_{k-1}$, then
$\link_{\Delta_k}(F)=\link_{\Delta_{k-1}}(F)$ and we are done by induction on $k$.
Else, as $F_k\cup F \in \Delta_{k-1}$ and $F_k,F\in \Delta$ we conclude that $F_k\cup F \in \Delta$.
By induction on $k$, $g^{(i)}(\link_{\Delta_{k-1}}(F))\geq 0$. Also, $\link_{\link_{\Delta_{k-1}}(F)}((F_{k}\setminus F))=
\link_{\Delta_{k-1}}(F\cup F_k)$. By construction of $\Delta_{k-1}$,
$F\cup F_k$ does not contain any of $F_1,\ldots,F_{k-1}$ (as $F\cup F_k \in \Delta_{k-1}$), hence the induction on $k$ says
$g^{(i)}(\link_{\Delta_{k-1}}(F\cup F_k))\geq 0$.
Thus, by Theorem \ref{thm:stellar} we conclude that $g^{(i)}(\link_{\Delta_k}(F))\geq 0$.
\end{proof}

\begin{remark}
For any subdivision $L$ of the $1$-simplex (say with $k$ interior points), $g^{(1)}(\diamondsuit(d,L))=\gamma(\diamondsuit(d,L)) \geq 0$.
This is known, and also follows from Theorem \ref{thm:stellar}, as $L$ is obtained by a sequence of $k$ stellar subdivisions at an edge.
\end{remark}

We now turn to arbitrary subdivisions $L$ of the $2$-simplex.

\begin{theorem}
\label{thm:2dim}
%Let $L$ be a subdivision of the $2$-simplex, where all the new vertices are in the interior.
If $\dim L=2$
then $\diamondsuit(d,L)\in \mathcal{HS}(2,d)$ and satisfies
$g^{(2)}(\diamondsuit(d,L))\geq 0$.
\end{theorem}
It is clear that $\diamondsuit(d,L)\in \mathcal{HS}(2,d)$. Below we state
and prove two generalizations of the second statement.

\begin{theorem}
\label{thm:delta_k}
%Let $L$ be a subdivision of the $2$-simplex, where all the new vertices
%are in the interior.
Let $\dim L=2$.
Then complexes $\Delta_k=\diamondsuit(d)_k$, arising
in the definition of a $\diamondsuit(d,L)$, satisfy
\begin{itemize}
\item[(i)] $g^{(2)}(\Delta_k)\geq 0$ and
\item[(ii)] $g^{(2)}(\link_{\Delta_k}(T_{k+1})\geq 0$
where $T_{k+1}$ is the $(k+1)$th $2$-simplex of $\Delta_0=\diamondsuit(d)$ that is subdivided.
\end{itemize}
\end{theorem}
\begin{proof}
We proceed by induction on $d$ and $k$ and instead of (ii) we will prove the
following stronger assertion:

(iii) If $F\in \Delta_0$ is of dimension $\geq 2$, and does not contain any of $T_1,\ldots,T_k$, then $g^{(2)}(\link_{\Delta_k}(F))\geq 0$.

The base case $d<2$ is trivial, and the case $k=0$ is clear as both $\Delta_0$ and $\link_{\Delta_0}(T_1)$ are boundary complexes of cross polytopes.
Let $m$ be the number of interior vertices in $L$, and $T$ be the $2$-simplex $L$ subdivides. By Euler's formula applied to the $2$-sphere $S=L\cup \{T\}$, one gets that the polynomial $f(L^0,t)$ counting faces of $L^0:=L\setminus \partial L$ satisfies $f(L^0,t)-t^3 = mt+3mt^2+2mt^3$.

Then
\begin{equation}\label{eq:L^0}
\begin{split}
f(\Delta_{k+1},t)&= f(\Delta_k,t) - t^3 f(\link_{\Delta_k}(T_{k+1}),t) + f(L^0,t) f(\link_{\Delta_k}(T_{k+1}),t)\\
&= f(\Delta_{k},t) + f(\link_{\Delta_k}(T_{k+1}),t) (mt+3mt^2+2mt^3).
\end{split}
\end{equation}
For any $(d-1)$-dimensional homology sphere $\Delta$, $h(\Delta,t)=(t-1)^d f(\Delta,\frac{1}{t-1})$, and combined with equation (\ref{eq:L^0}) we get
$$
h(\Delta_{k+1},t) = h(\Delta_{k},t) + mt(t+1) h(\link_{\Delta_k}(T_{k+1}),t).
$$
Note that the suspension $\Sigma\Delta$, i.e. the join of $\Delta$ with the two  points complex, has $h$-vector $(1+t)h(\Delta,t)$.
Thus
$$h(\Delta_{k+1},t) = h(\Delta_{k},t) + mt\cdot h(\Sigma\link_{\Delta_k}(T_{k+1}),t).
$$
By induction, $g^{(2)}(\Delta_k)\geq 0$ and
$g^{(2)}(\link_{\Delta_k}(T_{k+1}))\geq 0$,
so by Lemma~\ref{lem:missingResults} also
$g^{(2)}(\Sigma\link_{\Delta_k}(T_{k+1}))\geq 0$.
Thus,
$$
g^{(2)}(\Delta_{k+1},t)=
g^{(2)}(\Delta_k,t)+
mt \cdot g^{(2)}(\Sigma\link_{\Delta_k}(T_{k+1}),t)
$$
has only nonnegative coefficients, proving (i).

%Also, as in the proof of Theorem \ref{thm:stellar},
%$\link_{\Delta_{k+1}}(T_{k+2})$ equals the triangulation of $\link_{\diamondsuit(d)}(T_{k+2})$ induced by $L$ applied to all $2$-simplices $T_i$, $i<k+2$ such that $T_i\in \link_{\diamondsuit(d)}(T_{k+2})$. Thus, by induction on dimension we also have $g^{(2)}(\link_{\Delta_{k+1}}(T_{k+2}))\geq 0$.

To prove (iii),
if $F\cup T_k \notin \Delta_k$ then $\link_{\Delta_k}(F)=\link_{\Delta_{k-1}}(F)$ and we are done. Else, we treat different cases according to the cardinality of $F\cap T_k$:

{\bf Case $|F\cap T_k|=0$:}
Then $\link_{\Delta_k}(F)=(\link_{\Delta_{k-1}}(F))(T_k)$.
By induction on $k$, $g^{(2)}(\link_{\Delta_{k-1}}(F))\geq 0$, and
$g^{(2)}(\link_{\link_{\Delta_{k-1}}(F)}(T_k))=
g^{(2)}(\link_{\Delta_{k-1}}(F\cup T_k))
\geq 0$. Thus, by Theorem~\ref{thm:stellar} we are done.

{\bf Case $|F\cap T_k|=2$:}
Then $\link_{\Delta_k}(F)\cong \link_{\Delta_{k-1}}(F)$, via the isomorphism mapping the vertex $v\in \Int(T_k)$ adjacent to the edge $T_k\cap F$ to the vertex $T_k\setminus F$, and the other vertices to themselves. We are done by induction on $k$.

{\bf Case $|F\cap T_k|=1$:}
Let $v$ be the common vertex of $F$ and $T_k$, and let $P$ be the link of $v$ in the subdivision of $T_k$ induced by $L$ and the bijection $\phi:V(\partial L)\rightarrow T_k$. Then $P$ is a path, say with $s$ interior points (then $s\geq 1$).

Then $\link_{\Delta_k}(F)$ equals the subdivision of $\link_{\Delta_{k-1}}(F)$ induced by subdividing the edge $T_k\setminus F$ by $s$ interior points.
Thus,
$$f(\link_{\Delta_k}(F),t)= f(\link_{\Delta_{k-1}}(F),t) +
s t(1+t) f(\link_{\link_{\Delta_{k-1}}(F)}(T_k\setminus F),t),$$ equivalently,
$$h(\link_{\Delta_k}(F),t)= h(\link_{\Delta_{k-1}}(F),t) +
s t\cdot h(\link_{\link_{\Delta_{k-1}}(F)}(T_k\setminus F),t)
,$$ equivalently,
$$g^{(2)}(\link_{\Delta_k}(F),t)= g^{(2)}(\link_{\Delta_{k-1}}(F),t) +
s t\cdot g^{(2)}(\link_{\link_{\Delta_{k-1}}(F)}(T_k\setminus F),t)
.$$
By induction, both summands on the right hand side have nonnegative
coefficients (for the rightmost summand consider
$\link_{\Delta_{k-1}}(T_k\cup F)$), hence the left hand side has also
only nonnegative coefficients.
\end{proof}
Theorem~\ref{thm:2dim} is a special case of Theorem~\ref{thm:delta_k}
above, since $\diamondsuit(d,L)$ is the last complex
in the sequence of complexes $\Delta_1, \Delta_2,\ldots$. Before proving
the second generalization, let us make the following observation.
Obviously, for any homology sphere $\Delta\in \mathcal{HS}(i,d)$, we have
$g^{(i)}_0(\Delta)=1$, since we have $h_0(\Delta)=1$, the constant term of
$P_{d,i}(t)$ is $1$ and all other polynomials in the basis $B_{d,i}$ have zero
constant term. Therefore $g^{(i)}(\Delta)\geq 0$ holds (component-wise)
whenever the generalized $g$-polynomial given in (\ref{eq:gi})
has only real negative roots. Theorem~\ref{thm:2dim} is thus also a
consequence of the already shown Corollary~\ref{cor:Tcross}, Theorems
\ref{thm:gtch1} and \ref{thm:2dimRealRoot}, and of
Theorem~\ref{thm:Fstable} below.
\begin{theorem}
\label{thm:Fstable}
Let $\Delta$ be a homology sphere. Then the
following are equivalent:
\begin{itemize}
\item[(i)] the roots of $F(\Delta,t)$ are all real
  numbers in the interval $(-1,1)$;
\item[(ii)] the roots of $h(\Delta,t)$ are all real
  and negative;
\item[(iii)] the roots of  $g^{(2)}(\Delta, t)$ are all real numbers
  in the interval $[-1,0)$.
\end{itemize}
\end{theorem}
\begin{proof}
The equivalence of the first two statements may be shown by refining the
argument presented in \cite[Section  6]{Hetyei-Tch}. It was noted there
that the $F$-polynomial and the $h$-polynomial of $\Delta$ are connected by the formula
$$
(1-t)^d\cdot F\left(\Delta, \frac{1+t}{1-t}\right)=h(\Delta,t).
$$
The Preliminaries of \cite{Hetyei-Tch} remind of the well-known fact
that the map $\mu: x\mapsto t=(x-1)/(x+1)$ establishes a bijection
between the unit disk $|x|<1$ and the open left
$t$-halfplane. Using this bijection it is easy to show that the {\em
  Schur-stability} of $F(\Delta, x)$, defined as having all its roots
inside the unit disk $|x|<1$, implies the {\em Hurwitz-stability} of
$h(\triangle, t)$, defined as having all its zeros in the open left
$t$-halfplane. As noted in \cite[Proposition 6.4]{Hetyei-Tch}, the
converse is also true when the reduced Euler characteristic of $\Delta$
is not zero, which is the case for homology spheres. To arrive at the
presently stated equivalence we only need to observe that the
restriction of $\mu$ to the interval $(-1,1)$ establishes a bijection
between this interval and the set of all negative real numbers.

We are left two show the equivalence of the second and the third statement.
Directly from the definitions we have
$$
P_{k,2}(t)=
\begin{cases}
(1+t+t^2)^{k/2} & \mbox{for even $k$;}\\
(1+t+t^2)^{(k-1)/2} (1+t) & \mbox{for odd $k$.}\\
\end{cases}
$$
Using this formula it is easy to show
$$
h(\Delta,t)=
\begin{cases}
(1+t+t^2)^{d/2} g^{(2)}(\Delta, t/ (1+t+t^2))& \mbox{for even $d$;}\\
(1+t+t^2)^{(d-1)/2} (1+t) g^{(2)}(\Delta, t/ (1+t+t^2)) & \mbox{for odd $d$.}\\
\end{cases}
$$
Without loss of generality we may assume $d$ is odd, the case of even
$d$ being similar but simpler. Assume first all roots of
$g^{(2)}(\triangle,t)$ are real from $[-1,0)$, i.e., we have
$$
g^{(2)}(\Delta, t)=r(t-r_1)(t-r_2)\cdots (t-r_{(d-1)/2})
$$
for some positive real number $r$ and some negative real numbers $r_1, \ldots
r_{(d-1)/2}\in [-1,0)$. (The fact that $r$ is real and positive follows from
$g^{(2)}(\Delta)_0=1$.) Then we have
\begin{equation}
\label{eq:gh}
h(\triangle,t)=r(1+t)(t-r_1(1+t+t^2))\cdots (t-r_{(d-1)/2}(1+t+t^2)).
\end{equation}
The roots of $h(\triangle,t)$ are $-1$ and the roots of all quadratic equations
of the form $t-r_k(1+t+t^2)=0$, that is numbers of the form
\begin{equation}
\label{eq:st2}
s_k=\frac{(r_k-1)- \sqrt{(r_k-1)^2-4r_k^2}}{(-2r_k)}
\quad
\mbox{and of the form}
\quad
t_k=\frac{(r_k-1)+ \sqrt{(r_k-1)^2-4r_k^2}}{(-2r_k)}.
\end{equation}
Here $r_k-1$ is negative, the summand  $\sqrt{(r_k-1)^2-4r_k^2}$ is real
but strictly less than $|r_k-1|$, and the denominator $-2r_k$ is
positive. We obtain that each $s_k$ and $t_k$ is a negative real number.
To prove the converse, observe that Equation (\ref{eq:gh}) holds in
general for any homology sphere $\Delta$, with some complex roots $r_1,\ldots,
r_{(d-1)/2}$ and complex leading coefficient $r$. The roots of
$h(\triangle,t)$ are still $-1$ and the complex numbers $s_k$ and $t_k$
given by (\ref{eq:st2}). (Recall that taking the square root of a complex
number is unique up to sign, thus the pair $\{s_k,t_k\}$ is
well-defined.) Assuming that each $s_k$ and $t_k$ is a negative
real number, we obtain that each
$$
\frac{1-r_k}{r_k}=s_k+t_k
$$
is a negative real number and so each $r_k$ is a real number, belonging
to the set $(-\infty,0)\cup (1,\infty)$. Thus
$$
\sqrt{(r_k-1)^2-4r_k^2}=(-r_k)(t_k-s_k)
$$
is also a real number, and we must have $(r_k-1)^2-4r_k^2\geq 0$. This
is equivalent to $r_k\in [-1,1/3]$. The intersection of $[-1,1/3]$ with
$(-\infty,0)\cup (1,\infty)$ is the set $[-1,0)$.
\end{proof}
\begin{remark}
An analogous statement for $i=1$ was shown by Gal~\cite[Remark
  3.1.1]{Gal} who proved that, for a homology sphere $\Delta$, the
polynomial $h(\Delta,t)$ has only negative real roots if and only if
the same holds for $g^{(1)}(\Delta, t)$.
\end{remark}

\section*{Acknowledgments}
We are grateful to two
anonymous referees, whose suggestions greatly helped to improve the
presentation and substance of our manuscript.
% of their ideas.

This work was partially supported by a
grant from the Simons Foundation (\#245153 to G\'abor Hetyei).
Research of the second author was partially supported by Marie Curie
grant IRG-270923 and ISF grant 805/11.

\end{document}